\documentclass[preprint,12pt]{amsart}
\usepackage{amssymb, amscd}
\usepackage[arrow,matrix,curve]{xy}
\usepackage{fullpage,enumerate}

\newcommand{\beq}{\begin{equation}}
\newcommand{\eeq}{\end{equation}}
\newcommand\dirac{\slash\!\!\!\partial}

\newcommand\Ch{\operatorname{Ch}}

\newcommand\C{\mathbb C}

\newcommand\R{\mathbb R}

\newcommand\Z{\mathbb Z}
\newcommand\T{\mathbb T}

\newcommand\cM{\mathcal{M}}

\newcommand\maC{\mathcal{C}}

\newcommand\Ind{\operatorname{Ind}}

\newcommand\tr{\operatorname{tr}}

\newcommand\ch{{\mbox{ch}_{*}}}
\newcommand\cA{\mathcal{A}}

\newcommand\cF{\mathcal{F}}

\newcommand\cK{\mathcal{K}}
\newcommand\cG{\mathcal{G}}
\newcommand\cP{\mathcal{P}}
\newcommand\cL{\mathcal{L}}

\newcommand\cS{\mathcal{S}}
\newcommand\maS{\mathcal{S}}

\newcommand{\CC}{{\mathbb C}}

\newcommand{\NN}{{\mathbb N}}

\newcommand{\Q}{{\mathbb Q}}

\newcommand{\RR}{{\mathbb R}}
\newcommand{\TT}{{\mathbb T}}
\newcommand{\ZZ}{{\mathbb Z}}

\renewcommand\ch{\operatorname{ch}}
\renewcommand\Ch{\operatorname{Ch}}
\newcommand\Range{\operatorname{Range}}

\newtheorem{thm}{Theorem}[section]
\newtheorem{theorem}{Theorem}[section]
\newtheorem*{thm*}{Theorem}

\newtheorem{corollary}[thm]{Corollary}
\newtheorem*{cor*}{Corollary}

\newtheorem{conjecture}{Conjecture}
\newtheorem*{lemma*}{Lemma}
\newtheorem{problem}{Problem}

\newtheorem{lemma}[thm]{Lemma}
\newtheorem*{prop*}{Proposition}

\newtheorem{proposition}[thm]{Proposition}
\theoremstyle{definition}

\newtheorem{remark}[thm]{Remark}
\newtheorem*{defn*}{Definition}
\newtheorem{definition}{Definition}

\usepackage{color}
\definecolor{darkgreen}{cmyk}{1,0,1,.2}
\definecolor{m}{rgb}{1,0.1,1}
\definecolor{green}{cmyk}{1,0,1,0}

\definecolor{test}{rgb}{1,0,0}
\definecolor{cmyk}{cmyk}{0,1,1,0}

{

\begin{document}

\title{Gap-labelling conjecture with nonzero magnetic field}
\author{Moulay Tahar Benameur}
\address{Institut Montpellierain Alexander Grothendieck, UMR 5149 du CNRS, France}
\email{moulay.benameur@umontpellier.fr}

\author{Varghese Mathai}
\address{Department of Mathematics, University of Adelaide,
Adelaide 5005, Australia}
\email{mathai.varghese@adelaide.edu.au}

\begin{abstract}
Given a constant magnetic field on Euclidean space $\RR^p$ determined by a skew-symmetric $(p\times p)$ matrix $\Theta$,  and a $\Z^p$-invariant probability measure $\mu$ on the disorder set $\Sigma$ which is by hypothesis a Cantor  set, where the action is assumed to be minimal, the corresponding Integrated Density of States of any self-adjoint operator
affiliated to the twisted crossed product  algebra $C(\Sigma) \rtimes_\sigma \ZZ^p$, where $\sigma$ is the multiplier on 
$\ZZ^p$ associated to  $\Theta$, takes on values on spectral gaps in the {\em magnetic gap-labelling group}. The {\em magnetic frequency group} is defined as an explicit  
countable subgroup of $\RR$ involving Pfaffians of $\Theta$ and its sub-matrices. We conjecture that the magnetic gap labelling group is a subgroup of the magnetic frequency group.
 We give evidence for the validity of our conjecture in 2D, 3D, the Jordan block diagonal case and the periodic case in all dimensions.
\end{abstract}



\maketitle


\section{Introduction}


The gap-labelling theorem was originally conjectured by Bellissard \cite{Bellissard4} in the late 1980s. 
It concerns the labelling of gaps in the spectrum of a 
Schr\"odinger operator (in the absence of a magnetic field) by the elements of a subgroup of $\RR$ 
which results from pairing the $K_0$-group of the noncommutative analog for the Brillouin zone with the tracial state defined by the probability measure on the hull.
The problem arises in a mathematical version of solid state physics in the context of aperiodic tilings. 
Its three proofs, discovered independently by the authors of \cite{BO,KP,BBG} all concern the proof of a statement in K-theory.
Earlier results include the proof of the gap-labelling conjecture in 1D \cite{Bellissard5}, 2D \cite{BCL,VanEst} and in 3D \cite{BKL}. A more detailed account of the
history of gap-labelling theorems can be found in Appendix \ref{history}.

In the presence of a non-zero constant magnetic field in Euclidean space, the gap-labelling conjecture 
is much trickier to state, even though it was known to be the more interesting problem in spectral theory and in condensed matter physics since the 1980s, cf.~\cite{Bellissard85}. 
Here, we manage to give, for the first time, a
 precise formulation of conjectures for the magnetic gap-labelling group in all dimensions which encompass all previously known results. More precisely, in this paper we initiate the study of the gap-labelling group in the case of the magnetic Schr\"odinger operator on Euclidean space 
$\RR^p$ 
with disorder set a Cantor set $\Sigma$ under a non-zero magnetic field
$B=\frac{1}{2} dx^t \Theta dx$, where $\Theta$ is a $(p\times p)$ skew-symmetric matrix. We believe that proving (or disproving) our conjectures would constitute an important step in the understanding of aperiodic tilings  under a constant magnetic field. 
Given a $\Z^p$-invariant  probability measure $\mu$ on $\Sigma$, the corresponding Integrated Density of States of any self-adjoint operator
affiliated to the twisted crossed product  algebra $C(\Sigma) \rtimes_\sigma \ZZ^p$ takes  values on spectral gaps in an explicit  
countable subgroup of $\RR$
involving Pfaffians of $\Theta$ and its sub-matrices
 that we describe in Conjecture \ref{mainconj}, where $\sigma$ is the multiplier on 
$\ZZ^p$ associated to  $\Theta$. The physical interpretation of one side of our conjecture 
is a natural extension to the magnetic case of the notion of the group of frequencies studied in solid physics; see  \cite{Bellissard4}.  
In 2D, the magnetic gap-labelling group applies to the magnetic Schr\"odinger operators that are the Hamiltonians which are pertinent to the study of the
integer quantum Hall effect, cf.~\cite{BESB} and the bulk-boundary correspondence, cf.~\cite{Kell,MT}.

Upon defining the {\em magnetic gap-labelling group} and the {\em magnetic frequency group} in Definition \ref{defn:magnetic-groups}, our gap-labelling Conjecture \ref{mainconj}
states that for minimal actions of $\Z^p$ on a Cantor set, the magnetic gap-labelling group is a subgroup of the magnetic frequency group. Our 
gap-labelling Conjecture \ref{mainconj-minimal}
states that for {\em strongly} minimal actions of $\Z^p$ on a Cantor set, the magnetic gap-labelling group coincides with the magnetic frequency group.
Our main achievements in this paper, besides the precise statement of the conjectures,  are  complete solutions to the conjectures in the 2D case and also in the 3D case.
We also give other evidence  that our conjectures should hold in higher dimensions such as in the periodic case and the Jordan block diagonal case.   

The heart of our approach is a new index theorem, named the twisted index theorem for foliations, see subsection \ref{FoliatedIndex}. We also use the Baum-Connes conjecture with coefficients, which is known to be true for the relevant free abelian discrete group $\ZZ^p$ (cf.~\cite{BCH}). In addition, the integrality of {\em all} the components of the Chern character is needed to complete the proof of our magnetic gap-labelling conjecture, and is the trickiest part of the proofs of our theorems. This is in contrast to the proof of Bellissard's gap-labelling conjecture, where only the integrality of the {\em top} dimensional component of the Chern character is needed, and
it explains in a nutshell the difference in complexity of the two conjectures.
Direct cohomological computations in the 3D case enabled us to prove Conjecture \ref{mainconj-minimal}
(see Corollary \ref{3Dconj2})
for {\em strongly minimal systems}, a notion that is introduced in Definition \ref{stronglyminimal}. {We have included  the complicated combinatorics that proves  an independently interesting result in Theorem \ref{Morphism} and which makes possible a better
understanding of our magnetic gap-labelling conjectures.}
The proof of this theorem is a {\em tour de force} computation, and although our method extends to all dimensions, it only allowed us to deduce 
Conjecture \ref{mainconj-minimal} under an extra hypothesis on the corresponding clopen subsets. The strategy of proving the results in Section \ref{sect:3D}
and Section \ref{3D1} are outlined at the beginning of these sections.

In a forthcoming paper, we plan to weaken the hypotheses of Theorem \ref{Morphism} and to systematically study the magnetic gap-labelling group
in all higher dimensions.

It is worth pointing out  that  the  proofs mentioned earlier  of the Bellissard gap-labelling conjecture, which is the special case of the zero magnetic field, use the integrality of the top degree component of the Chern character for {\em{even}} dimensions. 
Since no published proof of this Chern-integrality result is known for general minimal $\Z^p$-actions on Cantor sets, we 
{\em do not} use it in the present paper. Notice however that the Chern-integrality condition is fulfilled for low dimensions (2D and 3D) and it was also proved in all even dimensions whenever the relevant $K$-theory is torsion-free, see  \cite{BO, BO-CRAS, BO-JFA}.  \bigskip

\noindent{\it Acknowledgements}. The authors  wish to thank Jean Bellissard for his encouraging comments 
and remarks on the first version of this paper and for providing the detailed history of the gap-labelling conjecture, 
which has now been incorporated into the introduction and also Appendix \ref{history}. The authors also thank Herv\'e Oyono-Oyono, Paolo Piazza and  Guo Chuan Thiang for helpful discussions. 
MB thanks the participant feedback on his talk at the {\em Workshop on KK-theory, Gauge Theory and Topological Phases}, at the Lorentz Center, Leiden, 6-10 March, 2017.
Particular thanks to M. Braverman, E. Prodan, R. Nest, J. Kellendonk, R. Meyer.
VM thanks the participant feedback on his talk at the {\em  2nd Australia-Japan Conference on Geometry, Analysis and their Applications}, Shirahama hotel, Japan, 
2-3 February, 2017. Particular thanks to T. Kato and S. Matsuo. {{Finally, the authors sincerely thank the referees for their in-depth reading of the first version of this paper and for their helpful comments and suggestions.}}\\

MB thanks the French National Research Agency for support via the ANR-14-CE25-0012-01 (SINGSTAR), and  VM thanks the Australian Research Council  for support
via the ARC Laureate Fellowship FL170100020.

\tableofcontents


\section{Magnetic Schr\"odinger operators}


We begin by reviewing the construction of  magnetic Schr\"odinger operators.
Consider Euclidean space $\RR^{p}$ equipped with its usual
metric $\sum_{j=1}^{p} dx_j^2$. 
Consider the uniform magnetic field 
$B = \frac{1}{2} dx^t \Theta dx =  \frac{1}{2} \sum_{j,k} \Theta_{jk} dx_j \wedge dx_k$, where $\Theta$ is a 
constant $(p\times p)$ skew-symmetric matrix.
The Euclidean group $G=\RR^{p} \rtimes SO(p)$ acts transitively on $\RR^{p}$ by affine transformations.
The torus $\TT^{p}$ can be realised as the
quotient of
$\RR^{p}$ by the action of its fundamental group $\ZZ^{p}$.

Let us now pick a 1-form $\eta$ such that $d\eta = B$. 
This is always possible since $B$ is a closed 2-form and 
$\RR^{p}$ is contractible.
We may regard $\eta$ as defining a connection
$\nabla = d+i\eta$ on the trivial line bundle $\cL$ over $\RR^{p}$, 
whose curvature is $i B$.
Physically we can think of $\eta$ as an electromagnetic vector potential for
the uniform magnetic field $B$ normal to $\RR^{p}$.
Using the Riemannian metric,  the Hamiltonian of an
electron in this field is given in terms of suitable units by
$$H_\eta = \frac 12\nabla^\dagger\nabla = \frac 12(d+i\eta)^\dagger(d+i\eta),$$
where $\dagger$ denotes the adjoint. 
In a real material this Hamiltonian would be modified by the addition of a {\em real-valued}
potential $V$, and called a magnetic Schr\"odinger operator $H_{\eta, V} = H_\eta +V$.
The spectrum of the unperturbed Hamiltonian $H_\eta$
for $\eta = \frac 12\sum \Theta_{jk} x_j  dx_k$ has been computed by physicists.
We record that it has discrete eigenvalues with infinite multiplicity. 
Any $\eta$ is cohomologous to $ \frac 12 \sum \Theta_{jk} x_j  dx_k$ since they both have
$B$ as differential, and forms differing by an exact form $d\phi$ give
equivalent models: in fact, multiplying the wave functions by
$\exp(i\phi)$ shows that the Hamiltonians for $\eta$ and $ \frac 12\sum \Theta_{jk} x_j dx_k$ are
unitarily equivalent. 
This equivalence also intertwines the $\ZZ^{p}$-actions so that the spectral
densities for the two models also coincide.
However, it is the perturbed Hamiltonian $H_{\eta,V} = H_\eta +V$ which
is the key, and the spectrum
of this is unknown for general $\ZZ^{p}$-aperiodic $V$. 
Set $\eta = \frac 12\sum \Theta_{jk} x_j  dx_k$ from now on.
For $\gamma  \in \ZZ^{p}$,
consider the function on $\RR^{p}$ given by   $\psi_\gamma(x) =  \frac{1}{2}\sum\Theta_{jk} \gamma_j x_k$. 
It satisfies $\gamma^*\eta-\eta = d \psi_\gamma$. Also, 
$\psi_\gamma(0)=0$ for all $\gamma  \in \ZZ^{p}$ 
and 
$\psi_\gamma(\gamma') = \frac{1}{2} \sum \Theta_{jk} \gamma_j \gamma'_k$ for $\gamma'  \in \ZZ^{p}$.
Define a projective unitary action $T^\sigma$ of $\ZZ^{p}$ on $L^2(\RR^{p})$ as follows.
\begin{align}
U_\gamma(f)(x) & = f(x-\gamma),\\
S_\gamma(f)(x) & = \exp(-2\pi i \psi_\gamma(x)) f(x),\\
T_\gamma^\sigma& =U_\gamma\circ S_\gamma.
\end{align}
Then the operators $T_\gamma^\sigma$, also known as {\em magnetic translations}, satisfy $T^\sigma_e=\rm{Id}, \,\, T^\sigma_{\gamma_1}
T^\sigma_{\gamma_2} = \sigma(\gamma_1, \gamma_2)T^\sigma_{\gamma_1\gamma_2}$,
where $\sigma(\gamma, \gamma') = \exp\left(-2\pi i \psi_\gamma(\gamma')\right)$  is a {\em multiplier}
on $\ZZ^{p}$ satisfying,
\begin{enumerate}
\item $\sigma(\gamma, e)=\sigma(e, \gamma)=1$ for all $\gamma \in \ZZ^{p}$;
\item $\sigma(\gamma_1, \gamma_2)\sigma(\gamma_1\gamma_2, \gamma_3)=\sigma(\gamma_1, \gamma_2\gamma_3)\sigma(\gamma_2, \gamma_3)$
for all $\gamma_j \in \ZZ^{p},\, j=1,2,3.$
\end{enumerate}
Note that with the above choices,  we also ensure the relation $\sigma(\gamma_1, \gamma_2) = \overline{\sigma(\gamma_2, \gamma_1)},$ and in particular $\sigma(\gamma, \gamma)=1,\, \forall\,\gamma\in \ZZ^p$.
An easy calculation shows that  $T^\sigma_\gamma H_\eta = H_\eta T^\sigma_\gamma$. Also, we shall assume that $V$ is aperiodic with hull equal to a Cantor set $\Sigma$ and we conclude that the magnetic Schr\"odinger operator
$H_{\eta,V} $ is also aperiodic with hull equal to $\Sigma$.

{{According to Bellissard's gap-labelling {{theorem}}, \cite{Bellissard4}, under usual  conditions on the aperiodic potential $V$, the $C^*$-algebra of observables is associated with a minimal dynamical system and is defined as follows. Fix $z\in \C\setminus \R$ and  consider the strong closure $X$ of the space of all conjugates of the resolvent $(H-z {\rm I})^{-1}$ under the magnetic translations $(T_a^\sigma)_{a\in \R^p}$. Then $X$ is independent of the choice of $z$ up to homeomorphism, and  it is a compact space with a minimal action of $\R^p$, through the same magnetic translations. The $C^*$-algebra of observables is then the twisted crossed product $C^*$-algebra $C(X)\rtimes_{\sigma} \R^p$. For the particular tilings we are interested in, and which include quasi-crystals \cite{BCL}, this latter $C^*$-algebra is Morita equivalent to a discrete twisted crossed product algebra $C(\Sigma) \rtimes_{\sigma} \ZZ^{p}$ for some Cantor space $\Sigma$. }}

 If $\lambda\in \RR$ is in a spectral gap
of $H_{\eta, V} $, then the Riesz projection $\chi_{(-\infty, \lambda]}(H_{\eta,V})$ can be expressed as $p_\lambda(H_{\eta, V})$ where $p_\lambda$ is a  smooth compactly supported
function which is identically equal to 1 in the interval $[{\inf spec} (H_{\eta, V}), \lambda]$, where ${\inf spec} (H_{\eta, V})$ denotes the bottom of the spectrum of the 
self-adjoint operator $H_{\eta, V}$ that is bounded below since $V$ is bounded below by our hypotheses. Assume also that 
the support of $p_\lambda$ is contained in the interval $[-\varepsilon +{\inf spec} (H_{\eta, V}), \lambda+ \varepsilon]$ for some $\varepsilon>0$.  
Then  $p_\lambda(H_{\eta, V})\in C(\Sigma) \rtimes_{\sigma} \ZZ^{p}\otimes\cK$
and therefore one obtains an element, 
$$[p_\lambda(H_{\eta, V})]\in K_0( C(\Sigma) \rtimes_{\sigma} \ZZ^{p}).$$
A standard assumption on the physical model is the {\em gap hypothesis}, which is that the Fermi level of the physical system described by $H_{\eta, V}$ lies in a spectral gap. We shall enunciate a precise conjecture generalising a famous conjecture 
of Bellissard in the absence of a magnetic field. We shall also give a complete proof of our conjecture in low dimensions, and give evidence for it in all dimensions.\\


\section{The magnetic gap-labelling group}\label{Conjecture}


In this section, we introduce and start our study of what we shall call the magnetic gap-labelling group.  Inspired by Bellissard's gap-labelling conjecture \cite{Bellissard4}, we show that the magnetic gap-labelling conjecture can be completely stated in the language of minimal totally disconnected  dynamical systems. More precisely, we assume that we are given $p$ commuting homeomorphisms $T=(T_j)_{1\leq j\leq p}$ of a Cantor space $\Sigma$ which preserve a Borel probability measure $\mu$. So $\Sigma$ is compact totally disconnected without isolated points and these homeomorphisms  then generate a minimal  action of the abelian free group $\Z^p$  so that $T_j$ corresponds to the action of the 
canonical basis vector $\psi_j\in \Z^p$.

The subgroup of the real line $\R$ which is generated by $\mu$-measures of clopen subspaces of $\Sigma$ is denoted $\Z[\mu]$. This is known as the group of frequencies of  the aperiodic potential associated with the quasi-crystal, i.e.  appearing in the Fourier expansion of that potential. It can also be seen as the image under (the integral associated with) the probability measure $\mu$ of $C(\Sigma, \Z)$, the group of continuous integer valued functions on $\Sigma$. That is,
$$
\Z[\mu] = \left\{\int_\Sigma f(z) d\mu(z) \Big| f\in C(\Sigma, \Z)   \right\} = \mu(C(\Sigma, \ZZ))
$$ 
Let $I$ be an ordered subset of $\{1,\ldots,p\}$ with an even number of elements, 
and let $C(\Sigma, \ZZ)_{\ZZ^{I^c}}$ denote the coinvariants of $C(\Sigma, \ZZ)$ under the action of the subgroup $\ZZ^{I^c}$
of $\ZZ^p$, where $I^c$ denotes the index set that is the complement to $I$. 
{{Let $\left(C(\Sigma, \ZZ)_{\ZZ^{I^c}}\right)^{\ZZ^I}$ 
denote the subset of $C(\Sigma, \ZZ)_{\ZZ^{I^c}}$ composed of those $\ZZ^{I^c}$-coinvariant classes in $C(\Sigma, \ZZ)_{\ZZ^{I^c}}$ which are invariant under the induced action of the subgroup $\Z^I$. }}
Define 
$$
\ZZ_I[\mu] = \mu\left(\left(C(\Sigma, \ZZ)_{\ZZ^{I^c}}\right)^{\ZZ^I}\right).
$$
{{Notice that 
$$
\ZZ_{\{1,\cdots, p\}}[\mu]={\ZZ \;  \subset  \; \ZZ_I[\mu] \; \subset  \; \ZZ[\mu]} = \ZZ_{\emptyset}[\mu].
$$}}

\subsection{Labelling the gaps}
Let $\sigma$ be a multiplier of $\Z^p$ which is associated with the skew symmetric matrix $\Theta$. The $\Z^p$ invariant probability measure $\mu$ yields a regular trace $\tau^\mu$ on the twisted crossed product $C^*$-algebra $C(\Sigma)\rtimes_\sigma \Z^p$, which is by fiat the operator norm completion of the $*$-algebra of compactly supported continuous functions
$C_c (\Z^p\times \Sigma)$ acting via the left regular representation on the Hilbert space $L^2(\Sigma, d\mu) \otimes \ell^2(\ZZ^d)$.
The trace $\tau^\mu$ is defined  on the dense subalgebra $C_c (\Z^p\times \Sigma)$ by the equality 
\beq\label{trace}
\tau^\mu (f)=\left\langle\mu, f_0\right\rangle  = \int_\Sigma f_0(z) d\mu(z), \quad \text{  where  } f_0: z\mapsto f(0, z), 
\eeq
with $0$ the zero element of $\Z^p$. Hence $\tau^\mu$ induces a trace
$$
\tau^\mu (= \tau^\mu_*)\; : \; K_0(C(\Sigma)\rtimes_\sigma \Z^p ) \longrightarrow \R.
$$
The $C^*$-algebra $C(\Sigma)\rtimes_\sigma \Z^p$ (or rather a Morita equivalent one) is a receptacle for the spectral projections onto spectral gaps of the magnetic Schr\"{o}dinger operator associated with our system, as explained earlier. Any gap in its spectrum  may therefore be labelled by the trace of the corresponding projection.

\begin{definition}\label{defn:magnetic-groups}
The range of the trace $\tau^\mu$
$$
 \Range (\tau^\mu) \; := \; \tau^\mu\left(K_0(C(\Sigma)\rtimes_\sigma \Z^p )\right) \; \subset \; \R
$$ 
is  what is natural to call  the {\em{magnetic gap-labelling group}}.  \\

The countable subgroup of $\R$ given by:
$$
\sum_{0\leq |I|\leq p} {\rm Pf}(\Theta_I)\Z_I[\mu],
$$
is what is natural to call the {\em magnetic frequency group}.
\end{definition}

It is easy to see the gap-labelling group (without magnetic field) is contained in the magnetic gap-labelling group.

We next formulate a conjecture that gives an explicit calculation of the magnetic gap-labelling group, and later give evidence for its validity.\\
\medskip

\begin{conjecture}[{\bf Magnetic gap-labelling conjecture: minimal actions}]\label{mainconj}
Let $\Sigma$ be a Cantor set with a minimal action of $\ZZ^p$ that preserves a Borel probability measure $\mu$.
Let $\sigma$ be the multiplier on $\ZZ^p$ associated to a skew-symmetric $(p\times p)$ matrix $\Theta$. \\
Then the magnetic gap-labelling group is contained in the magnetic frequency group.
\end{conjecture}

\medskip

{{So, more explicitly, we conjecture the following:}
\begin{enumerate}

{{\item If $p$ is even, then the magnetic gap-labelling group is contained in the countable subgroup of $\R$ given by:
$$
 \ZZ[\mu]  +  \sum_{0<|I|<p} {\rm Pf}(\Theta_I)\Z_I[\mu] +  {\rm Pf}(\Theta)\Z.
$$}}
{{\item If $p$ is odd, then the magnetic gap-labelling group is contained in the countable subgroup of $\R$ given by:
$$
 \ZZ[\mu]  +  \sum_{0<|I|\le p} {\rm Pf}(\Theta_I)\Z_I[\mu] .
$$}}
\end{enumerate}

\medskip

In both cases, $I$ is an ordered subset of $\{1,\ldots,p\}$ with an even number of elements, 
$\Theta_I$ denotes the skew-symmetric submatrix of 
$\Theta=(\Theta_{ij})$ with $i,j \in I$,\, and ${\rm Pf}(\Theta_I)$ denotes the Pfaffian of $\Theta_I$. 

We mention here the Pfaffian ${\rm Pf}(\Theta)$
was recognised first as the top degree coefficient in the range of the trace of the noncommutative torus associated to $\Theta$ in \cite{Elliott}, whereas the other terms were only given in terms of the coefficients of $\Theta$.  The complete set of coefficients in terms of 
the Pfaffians of submatrices of $\Theta$ is given over here for the first time in Proposition \ref{prop:periodic}, and is also due to \cite{MQ86} in another context.

\begin{remark}
When $p$ is even, we note that $C(\Sigma, \ZZ)^{\ZZ^p} = \ZZ$ since the $\ZZ^p$ action on $\Sigma$ is minimal,
which accounts for why the last term in part (1) above is a multiple of $\Z$. More precisely, when $p$ is even,
by the Baum-Connes map $\mu_\theta$ and the measured foliated twisted $L^2$-index theorem (see Section \ref{FoliatedIndex}),
$$
\tau_\mu (\mu_\theta(E)) = \sum_I {\rm Pf}(\Theta_I) \int_X dx_I \wedge ch(E)
$$
where $I$ is an ordered subset of $\{1,\ldots,p\}$ with an even number of elements, $dx_I$ is the differential form of degree equal to $|I|$
on the torus $\TT^p$ but lifted to $X= \Sigma \times_{\ZZ^p} \RR^p$, which is the total space of a fibre bundle over the torus $\TT^p$ with fibre 
the Cantor set $\Sigma$ and
$E$ is a vector bundle over $X$ whose Chern character is defined as in \cite{MooreSchochet}.
Consider the top degree term,
\begin{equation}\label{topterm}
{\rm Pf}(\Theta) \int_X dx_1 \wedge dx_2 \wedge... \wedge dx_p \, {\rm rank}(E). 
\end{equation}
Since the action of $\ZZ^p$ on the Cantor set $\Sigma$ is minimal, $X$ is therefore
a {\em connected} space  cf.~Lemma 3~\cite{BO-JFA}. So the rank of the vector bundle $E$ on $X$ is constant and
\eqref{topterm} now becomes
\begin{equation}
{\rm Pf}(\Theta) {\rm rank}(E) \rm {vol}(\TT^p) \mu(\Sigma) = {\rm Pf}(\Theta) {\rm rank}(E),
\end{equation}
since $\mu$ is a probability measure and the volume of the torus is normalised to equal 1. This implies that the last term in Conjecture \ref{mainconj} is always of the form ${\rm Pf}(\Theta) \ZZ$.

\end{remark}
\begin{remark}
We mention that for the dynamical systems arising from the cut-and-project quasi-crystals, the action is moreover free. One might thus tackle our conjecture under the extra freeness assumption but we do not make this assumption in the present paper, see again \cite{BO-JFA}. 
\end{remark}

In order to ensure the equality in the previous conjecture, it seems necessary to strengthen the minimality  assumption.
{{In particular, we mention the following interesting question:}}

\medskip

\begin{problem}
{{Under which (dynamical) condition, can one replace containment in Conjecture \ref{mainconj} by equality? 
}}
\end{problem}
\medskip

{{Of particular interest is the case of strongly minimal actions which are defined as follows. }}

\begin{definition}\label{stronglyminimal}
{{When $p$ is odd, an action of $\Z^p$ is said to be {\em strongly minimal} when for every generator $T_j$, the infinite cyclic group generated by it 
$\langle T_j\rangle$ acts minimally. When $p$ is even, an action of $\Z^p$ is said to be {\em strongly minimal} when for every pair of generators $(T_i, T_j)_{i\not = j}$, the $\Z^2$ group generated by the pair $\langle T_i, T_j \rangle$ acts minimally.}}
\end{definition}

\begin{remark}
{{Strongly minimal actions are clearly minimal. In the 2D case, minimal actions are strongly minimal. }}
\end{remark}

{{For $p\leq 3$, we prove that if the action of $\ZZ^3$ on  $\Sigma$ is strongly minimal, then the answer is yes. The strong minimality condition might though be too strong for 3D  (see Section \ref{3D1}), or even too weak for higher dimensions. It is surprising that despite the case of the zero magnetic field, where the missing containment relation always holds without any extra-condition, when the magnetic field is non-zero,  the answer to the above problem is not clear. As a starting point, we then state:
}}

\medskip
\begin{conjecture}[{\bf Magnetic Gap Labelling conjecture, strongly minimal actions}]\label{mainconj-minimal}
If the action if strongly minimal, then 
the magnetic gap labelling group coincides with the magnetic frequency group.
\end{conjecture}
\medskip

 \begin{remark}
{{As mentioned before, in the 2D case, minimal actions are strongly minimal and we prove Conjecture \ref{mainconj-minimal}
 in this case in Section \ref{sect:2D}, 
 in the 3D case in Section \ref{sect:3D} and also in the periodic case in Section \ref{Atheta}.}}
\end{remark}

\begin{remark}
There are several motivations for the magnetic gap-labelling conjecture, starting with Bellissard's gap-labelling conjecture \cite{Bellissard4}, 
the formula (see section 1 in \cite{MQ86})
$$
e^{\frac{1}{2}dx^t \Theta dx} = \sum_I {\rm Pf}(\Theta_I) dx^I
$$
in the notation above, the magnetic gap-labelling group when the potential of the magnetic Schr\"odinger operator is purely periodic which reduces  to the computation of the range of the trace on the noncommutative torus $A_\Theta$ using the formula above (see Section \ref{Atheta}),
and finally the proof of both the magnetic gap-labelling conjectures in 2D {{and  in 3D}}, that we give in this paper.
\end{remark}

\subsection{Some reductions of Conjectures \ref{mainconj} and \ref{mainconj-minimal}}

{{Let us first mention that it suffices to prove Conjecture \ref{mainconj}  and Conjecture \ref{mainconj-minimal} for $p$ even. }}

\begin{lemma}\label{lem:even}
{{If Conjecture \ref{mainconj} is true for some $p_0 \in \NN$, then it is also true for all $p \in \NN$ such that $p\le p_0$.
In particular, if Conjecture \ref{mainconj} is true for all $p \in 2\NN$, then it is true for all $p\in \NN$. 
If Conjecture \ref{mainconj} is true for all $p \in 2\NN +1$, then it is true for all $p\in \NN$. }}
\end{lemma}

\begin{proof}\
By an easy trick, see for instance \cite{BO-JFA}, given a minimal Cantor system of dimension $p$ as above we may embed $\Z^p$ in the group $\Z^{p+1}$ by using $\iota (n)=(n, 0)$ and notice that $\Z^{p+1}$  acts minimally on $\Sigma$ by using  the $\Z^p$-action of the projection onto the $p$-first factors. We get in this way a new minimal Cantor system of dimension $p+1$. The multiplier $\sigma$ on $\Z^p$ then gives rise to the multiplier  $\iota_*\sigma$ on $\Z^{p+1}$ which is given by
$$
\iota_*\sigma ((n, n_{p+1}), (m, m_{p+1}))=  \sigma (n, m).
$$
Then $\iota_*\sigma$ is associated with the skew symmetric matrix 
$$
\iota_*\Theta = \left(   \begin{array}{cc} \Theta & 0 \\ 0 & 0 \end{array}\right).
$$
Now, the above inclusion $\iota$ yields an inclusion $i: C(\Sigma)\rtimes_\sigma \Z^p\hookrightarrow C(\Sigma)\rtimes_{\iota_*\sigma} \Z^{p+1}$ and it is easy to check that  the following  diagram commutes
$$
\begin{CD}
K_0 (C(\Sigma)\rtimes_\sigma \Z^p) @> {i_*}  >> K_0 (C(\Sigma)\rtimes_{\iota_*\sigma} \Z^{p+1})\\
@V{\tau^\mu_* }VV     @VV {\tau^\mu_*} V\\
\R @> {=} >> \R
\end{CD}
$$
 Hence if we assume that Conjecture \ref{mainconj} is true for $p+1$, then we get 
\beq\label{eqn:contain}
\tau^\mu_* \left(K_0(C(\Sigma)\rtimes_\sigma \Z^p)\right)  = (\tau^\mu_*\circ i_*) \left(K_0(C(\Sigma)\rtimes_{\iota_*\sigma} \Z^{p+1})\right) \subset \sum_{0\leq \vert I\vert \leq p+1}  {\rm Pf}((\iota_*\Theta)_I) \Z_I[\mu].
  \eeq
  But from the expression of $\iota_*\Theta$, we see that if $I$ contains $p+1$ then ${\rm Pf} ((\iota_*\Theta)_I) = 0$ while for $I\subset \{1, \cdots, p\}$, we have ${\rm Pf}((\iota_*\Theta)_I)={\rm Pf}(\Theta_I)$. 
  \end{proof}

The similar proof shows:

\begin{lemma}\label{Even}
{{If Conjecture \ref{mainconj-minimal} is true for some $p_0 \in 2\NN$, then it is also true for $p_0-1$.
In particular, if Conjecture \ref{mainconj-minimal} is true for all $p \in 2\NN$, then it is true for all $p\in \NN$.  }}
\end{lemma}

\begin{proof}
We apply the same proof as for the previous lemma but replace the containment in \eqref{eqn:contain} with equality
and notice that $p=p_0-1$ is odd and hence we consider a strongly minimal system in the odd case. This means that every generator $T_j$ for $j=1,\cdots, p_0-1$ acts minimally. But then adding a trivial action of $\Z$ as above, we get again a strongly minimal action for $p=p_0$ even. 
\end{proof}

We make some observations now to help investigate the conjectures. 
In comparison with the non-magnetic case, we expect the inclusion of the magnetic gap-labelling group into the magnetic frequency group to be the easy half of the conjecture. Recall that  $\Z^{J}$ denotes the subgroup of $\Z^p$ which corresponds to a given ordered subset $J=\{j_1 < \cdots < j_{\vert J\vert}\}$ of $\{1, \cdots, p\}$ where we put zero components for the entries corresponding to the complement set $J^c$. So the action of $\Z^J$ on $\Sigma$ is generated by the homeomorphisms $(T_j)_{j\in J}$. We denote for {{$\alpha=(\alpha_1, \cdots, \alpha_{\vert J\vert}) \in \Z^{\vert J\vert}$ by $T^\alpha$ the homeomorphism of $\Sigma$ which is given by 
$$
T^\alpha := \Pi_{1\leq \ell\leq \vert J\vert}\;\;  T_{j_\ell}^{\alpha_\ell} = T_{j_1}^{\alpha_1}\circ \cdots \circ T_{j_{\vert J\vert}}^{\alpha_{\vert J\vert}}.  
$$}}
Notice that the $\Z$-module $C(\Sigma, \Z)$ is generated by the characteristic functions of the clopen subsets of $\Sigma$, however an identification of $C(\Sigma, \Z)$ with a $\Z$-module ($\Z$-measures) constructed out of  the Boolean algebra $\maS$ of clopen subsets of $\Sigma$ is not helpful.  For $J=I^c$ with $I$ an ordered multi-index set of even length as in the statement of the conjecture, one expects to give a simple description of the class in $C(\Sigma, \Z)_{\Z^{I^c}}$ of any clopen set $\Lambda$ so that the $\Z^I$ invariance can be exploited.
So, in order to show for instance that the magnetic frequency group (RHS) is contained in the magnetic gap-labelling group, one needs to show that for any such $\Lambda$, the real number ${\rm Pf}(\Theta_I)\times \mu (\Lambda)$ belongs to the magnetic gap-labelling group (LHS).  

This inclusion will be 
given explicitly in Section \ref{3D1} when $\Lambda$ has a convenient presentation. We shall only give this construction in the $3D$ case where we succeeded to prove both our conjectures. Although the construction is already  combinatorially involved,  it turns out to be feasible by direct inspection. We expect that our proof will help to deduce an explicit construction of the easy half of the conjecture for {{general algebraic combinations}} of clopens satisfying the invariance property in the coinvariants, that is whose class belongs to $\left( C(\Sigma, \Z)_{\Z^{I^c}}\right)^{\Z^I}$.

  Another important observation is that the $K$-theory group of the twisted crossed product algebra is isomorphic to that of the untwisted one.  More precisely,  the following is probably  known to experts, but we give the short proof. Let $X= \Sigma \times_{\ZZ^p} \RR^p$ be the suspension of $\Sigma$, that is the quotient of the cartesian product $\Sigma\times \R^p$ under the diagonal action of $\Z^p$. The additive group $\R^p$ acts on $X$ and this action yields a lamination of $X$ and the action groupoid $X\rtimes \R^p$.  Recall that the groupoid $X\rtimes \R^p$ is strongly Morita equivalent to the groupoid $\Sigma\rtimes \Z^p$, see \cite{Rieffel2, Green}, see also \cite{KP}.

  \begin{theorem}[Twisted Connes-Thom isomorphism]\label{twistedCT}
  Let $\Sigma$ be a Cantor set with an action of $\ZZ^p$ and let $X= \Sigma \times_{\ZZ^p} \RR^p$.
Let $\sigma$ be the multiplier on $\ZZ^p$ associated to a skew-symmetric $(p\times p)$ matrix $\Theta$.
 Then
 $$
K^p(X)  \cong  K_0(C(\Sigma)\rtimes_\sigma\Z^p).
 $$
 \end{theorem}
 \begin{proof}
By the strong Morita equivalence of the groupoids $X\rtimes \R^p$  and $\Sigma\rtimes \Z^p$,
we see that the crossed product $C^*$-algebra $C(X)\rtimes\RR^p$,  is strongly Morita equivalent 
to the crossed product $C^*$-algebra $C(\Sigma)\rtimes\Z^p$.
In particular, using Connes-Thom isomorphism \cite{Connes81,FackSkandalis} , one has
\beq\label{eqn:Kiso2}
K^p(X)  \cong K_0(C(X)\rtimes\RR^p) \cong K_0(C(\Sigma)\rtimes \Z^p)
 \eeq
 For $t\in [0,1]$, let $\sigma_t$ denote the multiplier corresponding to the $(p\times p)$ skew symmetric
matrix $t\Theta$. More precisely, 
$$\sigma_t(\gamma, \gamma') = \exp\left(2\pi \sqrt{-1} t\sum_{j<k} \Theta_{jk} \gamma_j \gamma'_k\right), \qquad \text{where}\quad \gamma, \gamma'  \in \ZZ^{p}.$$
Now $\{C(\Sigma) \rtimes_{\sigma_t} \ZZ^p: t\in [0,1]\}$ is a homotopy of twisted crossed products in the
sense of section 4, \cite{PR2}, where $\sigma_0=1$ and $\sigma_1=\sigma$. By Theorem 4.2 in \cite{PR2}, we deduce that 
 \beq\label{eqn:Kiso3}
 K_0(C(\Sigma)\rtimes \Z^p) \cong K_0(C(\Sigma)\rtimes_\sigma \Z^p).
 \eeq
 The Theorem follows from \eqref{eqn:Kiso2} and \eqref{eqn:Kiso3}.
 \end{proof}

\subsection{The measured twisted foliated index theorem}\label{FoliatedIndex}

Here we sketch a proof of a special case of the measured twisted index theorem that we need in this paper.
It is a twisted analogue of a special case of the index theorem in Benameur-Piazza \cite{BenameurPiazza}.
The general case will be treated in \cite{BenameurMathaiPreprint}. 

Let $\rho: \Z^p \longrightarrow {\rm Homeo}(\Sigma)$ denote the minimal action of $\Z^p$ on $\Sigma$. 
We suppose that $\mu$ is an invariant measure on $\Sigma$ and that $p$ is even. 
Then the suspension $X= \RR^p \times_{\Z^p}\Sigma$ is a compact foliated space with transversal the Cantor set $\Sigma$,
and with invariant transverse measure induced from $\mu$ (cf. Proposition 2.2 \cite{KP}).  The monodromy
groupoid is:
\beq
\cG = (\RR^p \times \RR^p \times \Sigma)/\ZZ^p.
\eeq
Then the twisted monodromy groupoid $C^*$-algebra, $C^*(X, \cF, \sigma)$ consists of the operator norm closure of 
continuous functions $k$ satisfying the following conditions:
\begin{enumerate}
\item $k \in C(\RR^p \times \RR^p \times \Sigma)$;
\item $k(x.\gamma, y.\gamma, \vartheta.\gamma) = e^{i\varphi_\gamma(x)} k(x,y,\vartheta)  e^{-i\varphi_\gamma(y)} $ for all $x,y\in \RR^p, \, \vartheta\in\Sigma, \, \gamma\in\ZZ^p$;
\end{enumerate}
Define the transverse trace as,
\beq
\tau_\mu(k) = \int_X k(x,x,\vartheta)  d\mu(\vartheta) dx.
\eeq
It easily extends to matrix valued kernel functions, by composing with the pointwise matrix trace.
Next define the continuous functions $\varphi_\gamma(x, \gamma)$.
Let $\Theta(\vartheta)$ be a skew-symmetric $p\times p$ matrix that is a continuous function of $\vartheta \in \Sigma$ 
and such that $\Theta(\vartheta.\gamma) = \Theta(\vartheta)$ for all $\vartheta \in \Sigma, \, \gamma\in \ZZ^p$.
Since the action of $\ZZ^p$ on $\Sigma$ is assumed to be minimal, it follows that $\Theta=\Theta(\vartheta)$ is independent of $\vartheta  \in \Sigma$.

Set $B = \frac{1}{2} dx^t\Theta dx,$ which is a closed 2-form on $\RR^p \times \Sigma$ (which is independent of $\vartheta  \in \Sigma$) 
satisfying $\gamma^*B=B$ for all $\gamma \in \ZZ^p$.
Since $B=d\eta$ where for instance $\eta = \sum_{j<k} \Theta_{jk} x_j dx_k$, we get $0=d(\gamma^*\eta - \eta)$. Since $\RR^p$ is simply-connected,
we see that $\gamma^*\eta - \eta= d\phi_\gamma$, where $\phi_\gamma$ is a smooth function on $\RR^p \times \Sigma$ (which is independent of $\vartheta \in \Sigma$). We normalise it 
so that $\phi_\gamma(0)=0$ for all $\gamma\in\ZZ^p$.

Consider functions $f \in L^2(\RR^p \times \Sigma; dxd\mu)$ and bounded operators on it defined as follows,
\begin{enumerate}
\item $S_\gamma f(x,\vartheta) = e^{i\varphi_\gamma(x)} f(x, \vartheta)$;
\item $U_\gamma f(x,\vartheta) = f(x.\gamma, \vartheta.\gamma)$.
\end{enumerate}
Then for all $\gamma \in \ZZ^p$, the bounded operators $T_\gamma=U_\gamma \circ S_\gamma$ satisfy the relation
\beq
T_{\gamma_1} T_{\gamma_1} = \sigma(\gamma_1,\gamma_2)\, T_{\gamma_1\gamma_2}
\eeq
where $ \sigma(\gamma_1,\gamma_2)= \phi_{\gamma_1}(\gamma_2)$ is a multiplier on $\ZZ^p$.

Let $\dirac$ denote the Dirac operator on $\RR^p$ and $\nabla = d+i\eta$ the connection on the trivial line
bundle on $\RR^p$, $\nabla^E$ the lift to $\RR^p \times \Sigma$ of a connection on a vector bundle $E\to X$. 
Consider the twisted Dirac operator along the leaves of the lifted foliation,
\beq
D= \dirac\otimes\nabla\otimes\nabla^E : L^2(\RR^p \times \Sigma, \cS^+\otimes E)\longrightarrow L^2(\RR^p \times \Sigma, \cS^-\otimes E).
\eeq
Then one computes that $T_\gamma \circ D= D\circ T_\gamma$ for all $\gamma\in\ZZ^p$.
The heat kernel $k(t,x,y,\vartheta)$ of $D$, although  not compactly supported, 
can be shown as usual to belong to the $C^*$-algebra of the foliation. More precisely, $k(t,x,y,\vartheta) \in C^*(X, \cF, \sigma)\otimes \cK.$ 
For $t>0$, define the idempotent 
$$e_t(D)\in  
M_2(C^*(X, \cF, \sigma)\otimes{\mathcal K})$$ as follows:
$$
e_t(D) = \begin{pmatrix}
e^{-tD^-D^+} & \displaystyle
e^{-{\frac{t}{2}}D^-D^+}\frac{(1-e^{-tD^-D^+})}
  {D^-D^+} D^+ \\[+11pt]
e^{-{\frac{t}{2}}D^+D^-}{D^+} &
        1- e^{-tD^+D^-}
\end{pmatrix},
$$
It is the analogue of the Wassermann idempotent, see e.g. \cite{ConnesMoscovici}. Then the $C^*(X, \cF, \sigma)$-{\em twisted analytic index}
is defined as
\begin{equation}\label{bc}
{\rm Index}_{C^*(X, \cF, \sigma)}(D^+)   = 
[e_t(D)] - [E_0]  \in K_0(C^*(X, \cF, \sigma)),
\end{equation}
where $t>0$ and $E_0$ is the idempotent
$$
E_0 = \begin{pmatrix} 0 & 0 \\ 0 &
        1
\end{pmatrix}.
$$
Finally, using the fact that $C^*(X, \cF, \sigma)$ and $C(X)\rtimes_\sigma \RR^p$ are isomorphic,
this constructs  a {\em twisted foliated index map}, {{generalizing}} Section 2, \cite{MarcolliMathai}
and also \cite{Connes82}, \cite{MooreSchochet}
$$
{\rm Index}_{C^*(X, \cF, \sigma)}: K^0(X) \longrightarrow K_0(C(X)\rtimes_\sigma \RR^p) \text{ as } [E] \longmapsto {\rm Index}_{C^*(X, \cF, \sigma)}(D^+).
$$
The twisted measured index of $D^+$ is then by definition the $\tau^\mu$-trace of the index class, a real number.

\begin{theorem}
Under the previous assumptions, the measured index of $D^+$ is given by the formula
$$
 \tau_\mu ({\rm Index} (D^+))  =    \sum_I    {\rm Pf}(\Theta_I) \int_{X} dx_I\wedge  \ch (F_E) d\mu(\vartheta).
$$
Here $I$ runs over subsets of $\{1,\ldots,p\}$ with an even number of elements, $dx_I$ is the differential form of degree equal to $|I|$
on the torus $\TT^p$ which is lifted to $X= \Sigma \times_{\ZZ^p} \RR^p$, and $\Theta_I$ is the skew-symmetric submatrix of 
$\Theta=(\Theta_{ij})$ with $i,j \in I$. Finally, $F_E$ denotes the curvature of the connection $\nabla^E$ on the vector bundle $E$ over $X$.
\end{theorem}

When $p$ is odd, a similar construction gives the odd measured index formula. 

\begin{proof}
A standard McKean-Singer type argument shows that the supertrace 
\begin{equation}\label{eqn:supertrace}
\tau_\mu({\rm tr}_s (k(t,\cdots))) = \tau_\mu(e^{-tD^-D^+}) - \tau_\mu(e^{-tD^+D^-}) = \tau_\mu ({\rm Index}_{C^*(X, \cF, \sigma)}(D^+))
\end{equation}
is independent of $t>0$ and  
represents the measured twisted foliated index. To be self-contained, we outline the argument. First we show that 
equation \eqref{eqn:supertrace} is independent of $t>0$. The heat operator $e^{-tD^2}$ can be differentiated with respect to $t$, since
 $\frac{d}{dt} e^{-tD^2}$ is a smoothing operator equal to $-D^2  e^{-tD^2}$, and therefore 
 $$
 \frac{d}{dt} \tau_\mu^s(e^{-tD^2}) = -  \tau_\mu^s(D^2  e^{-tD^2}) = -\frac{1}{2}\tau_\mu^s([D, D  e^{-tD^2}])=0
 $$
where the last equality holds since $D$ is an odd operator. Here $ \tau_\mu^s$ denotes the graded version of the trace $ \tau_\mu.$ This shows that 
$\tau_\mu^s(e^{-tD^2})$ is independent of $t>0$.

To complete the proof, we need to show that the smoothing kernel $k(t,\cdots)$ of $e^{-tD^2}$ converges to the smoothing kernel
of the projection $P$ to the nullspace of $D$, uniformly on compact subsets as $t\to \infty$. But this is identical to the argument given in
the proof of Proposition 15.11 in \cite{Roe98}.

Now, using the expression of the trace as an integral over the fundamental domain $\Sigma\times (0, 1)^p$ with respect to the product measure $\mu\otimes dvol$ on $\Sigma\times \R^p$ and applying the standard local index method, see \cite{BGV},  we obtain
\begin{align}
\lim_{t\downarrow 0} \tau_\mu({\rm tr}_s (k(t,\cdots))) &= \frac{1}{(2\pi)^p}  \int_{\Sigma\times (0, 1)^p}  \exp\left(\frac{1}{2} dx^t\Theta dx\right) \wedge \ch (F_E) d\mu(\vartheta),\\
&= \frac{1}{(2\pi)^p}   \sum_I    {\rm Pf}(\Theta_I) \int_{X} dx_I\wedge  \ch (F_E) d\mu(\vartheta).
\end{align}
Here $I$ runs over subsets of $\{1,\ldots,p\}$ with an even number of elements, and $\Theta_I$ denotes the skew-symmetric submatrix of 
$\Theta=(\Theta_{ij})$ with $i,j \in I$. Observe that $\frac{i}{2} dx^t\Theta dx$ is the curvature of the connection $\nabla$, and that 
$ \exp\left(\frac{1}{2} dx^t\Theta dx\right)$ is the Chern character of $\nabla$.\\
\end{proof}


\section{Magnetic gap-labelling group for periodic potentials}\label{Atheta}


Let $\Lambda[dx]=\Lambda[dx_1, \ldots, dx_p]$ denote the exterior algebra with generators $dx_1, \ldots, dx_p$. 
It has basis the monomials $dx_I= dx_{i_1}, \ldots, dx_{i_p}, \,\, I=\{i_1,\ldots,i_p\},\,\, i_1<\cdots<i_p$.
Given a skew-symmetric matrix $\Theta$, we can associate a quadratic element $\frac{1}{2}dx^t \Theta dx$ in $\Lambda[dx]$.
Here $dx$ is the column vector with entries $dx_j$ and $dx^t$ is the row vector with the same entries. Then the Gaussian 
$e^{\frac{1}{2}dx^t \Theta dx} $ can be expressed in terms of the Pfaffians, see section 1 in the paper of Mathai-Quillen \cite{MQ86},
\beq\label{exp}
e^{\frac{1}{2}dx^t \Theta dx} = \sum_I {\rm Pf}(\Theta_I) dx_I
\eeq
where $I$ runs over subsets of $\{1,\ldots,p\}$ with an even number of elements, and $\Theta_I$ denotes the submatrix of 
$\Theta=(\Theta_{ij})$ with $i,j \in I$, which is clearly also skew-symmetric. 

Let $\tau: A_\Theta\to \CC$ denote the von Neumann
trace. Then the magnetic gap-labelling group for magnetic Schr\"odinger operators $H_{\eta, V}$ where $V$ is periodic
is given by Proposition \ref{prop:periodic}, which is originally due to Elliott in \cite{Elliott}, although in that paper, only the Pfaffian ${\rm Pf}(\Theta)$
was {{recognized, whereas the other terms were only given in terms of the coefficients of $\Theta$, but were not recognized}} as the 
the Pfaffians of submatrices of $\Theta$ as neatly given over here, and was originally due to \cite{MQ86} in another context. Moreover our proof is different, being based on index theory and algebraic topology.

\begin{proposition}[Magnetic gap-labelling group for periodic potentials]\label{prop:periodic}
The range of the trace on the K-theory of $A_\Theta$ is:
\begin{enumerate}
\item If $p$ is even, then
$$
\tau (K_0(A_\Theta) )= \Z +\sum_{0<|I|<p} {\rm Pf}(\Theta_I) \Z  + {\rm Pf}(\Theta)\Z,
$$
\item If $p$ is odd, then
$$
\tau (K_0(A_\Theta) )= \Z+\sum_{0<|I|<p} {\rm Pf}(\Theta_I) \Z,
$$
\end{enumerate}
where $I$ runs over subsets of $\{1,\ldots,p\}$ with an even number of elements, and $\Theta_I$ denotes the submatrix of 
$\Theta=(\Theta_{ij})$ with $i,j \in I$.
\end{proposition}
\begin{proof}
Since the Baum-Connes conjecture with coefficients is true for $\Z^p$, it follows that the twisted Baum-Connes conjecture  is true for $\Z^p$.
$$
\mu_\Theta : K^p(\TT^p) \stackrel{\sim}{\longrightarrow} K_0(A_\Theta) 
$$
is an isomorphism. 
Then by Appendix A  and the twisted L$^2$-index theorem \cite{Mathai99} and equation \eqref{exp}, one has 
\begin{align*}
\tau(\mu_\Theta(\xi)) &= \int_{\TT^p} e^{\frac{1}{2}dx^t\Theta dx} \wedge \Ch(\xi) \\
&=  \sum_I    {\rm Pf}(\Theta_I) \int_{\TT^p} dx_I\wedge \Ch(\xi)_{I^c} 
\end{align*}
where $I^c$ is the index that is complement to $I$ and $\Ch(\xi)_{I^c} $ denotes the component of the Chern character $\Ch(\xi)$
of the vector bundle $\xi$
containing $dx_{I^c}$.
Since the Chern character is integral on the torus $\TT^p$, the result follows by varying $\xi$ over all K-theory classes.

\end{proof}

\begin{remark}
In fact, Elliott in \cite{Elliott}, proved more in that he computed the whole Connes-Chern character in terms of the 
canonical action of the $p$-dimensional torus. 
The special case when $p=2$ was partially due to Rieffel, \cite{Rieffel}
and partially to Pimsner-Voiculescu \cite{PV,PV2}.\\
\end{remark}


\section{Computation of the 2D magnetic gap-labelling group}\label{sect:2D}
 
 
 We now compute the magnetic gap-labelling group in the easiest case of $p=2$. 
Let again $\ZZ^2\curvearrowright \Sigma$ be a minimal action with invariant
probability  measure $\mu$. Let $\sigma$ be a multiplier on $\ZZ^2$. Then the group cohomology class of $[\sigma]
 \in H^2(\ZZ^2; \RR/\ZZ) \cong \RR/\ZZ$ can be identified with a real number $\theta, \, 0\le \theta<1$. More precisely,
 we take $\sigma = e^{2\pi i \theta\omega}$ where $\omega$ is the standard symplectic form on $\ZZ^2$.
 
 Notice that an easy inspection shows that the natural inclusion
                $A_\theta =C^{*}_{r}(\Z^2, \sigma)\stackrel{\iota}{\hookrightarrow} C(\Sigma)\rtimes_\sigma\Z^2\,$, takes the 
Rieffel projection $\cP_\theta$ \cite{Rieffel} to the projection $\iota(\cP_\theta)$ in $C(\Sigma)\rtimes_\sigma\Z^2\,$ which generates a $\Z$
factor in $K_0(C(\Sigma)\rtimes_\sigma\Z^2)$. Indeed let $\mu:C(\Sigma)\to \CC$ and $\tau^\mu:C(\Sigma)\rtimes_\sigma\Z^2\to \CC$ be the traces
induced by $\mu$, see equation \eqref{trace}, and by the same notation the maps induced on K-theory. Then we have,

\begin{lemma}\label{2Dtrace}
Let $\tau: A_\theta \to \CC$ denote the von Neumann trace. Then there is a  commutative diagram,
$$
\xymatrix{
A_\theta\,  \ar@{^{(}->}[r]^{\iota\qquad} \ar[d]^{\tau} & C(\Sigma)\rtimes_{\sigma} \ZZ^2 \ar[d]^{\tau_\mu} \\
\CC \ar[r]^{=}&\CC }   
$$
In particular, we see that $\theta = \tau(\cP_\theta) = \tau_\mu(\iota(\cP_\theta))$.
\end{lemma}

\begin{proof}

Let $f(\gamma) \in A_\theta$. Then by equation \eqref{trace}, we see that 
$$
\tau(f) = f(0), \quad \tau^\mu(\iota(f))=  \tau^\mu(f)= \int_\Sigma f(0) \, d\mu(z) = f(0),
$$
since $\mu$ is a probability measure on $\Sigma$.
\end{proof}

 \begin{theorem}\label{2DKthy}
 $$
 K_0(C(\Sigma)\rtimes_\sigma\Z^2)\simeq C(\Sigma, \ZZ)_{\ZZ^2} \oplus \Z[\iota(\cP_\theta)].
 $$
 where $C(\Sigma, \ZZ)_{\ZZ^2}$ denotes the space of coinvariants.
 \end{theorem}
 \begin{proof}
The computation of $K_0(C(\Sigma)\rtimes\Z^2)$ in the untwisted case was carried out 
 in \cite{BCL},
$$
        K_0(C(\Sigma)\rtimes\Z^2)\simeq C(\Sigma, \ZZ)_{\ZZ^2} \oplus \Z.
$$
By Theorem \ref{twistedCT} above, we have
 $$
K^0(X) \cong K_0(C(\Sigma)\rtimes\Z^2) \cong K_0(C(\Sigma)\rtimes_\sigma\Z^2).
 $$
Therefore the result follows.

\end{proof}

\begin{remark}
The computation in this 2D case is thus identical to the untwisted case and we only replace the Bott projection by the Rieffel projection.
\end{remark}

\begin{corollary}[Magnetic gap-labelling in 2 dimensions]\label{2Dgapthm}
$$
\tau^\mu(K_0(C(\Sigma)\rtimes_\sigma\Z^2)) = \Z[\mu] + \Z \theta. 
$$
\end{corollary}

\begin{proof} We use Theorem \ref{2DKthy}.
Now $\tau^\mu(C(\Sigma, \ZZ)_{\ZZ^2}) = \mu(C(\Sigma, \ZZ))=\ZZ[\mu]$ since the measure is invariant. 
By Lemma \ref{2Dtrace}, $\tau_\mu(\iota(\cP_\theta)) = \theta$.

\end{proof}

We next give another proof of this result, using index theory, in order to help investigate the general case.

\begin{proof}[2nd Proof]
By Theorem \ref{twistedCT},
$$
\mu_\theta : K^0(X) \stackrel{\sim}{\longrightarrow} K_0(C(\Sigma)\rtimes_\sigma\Z^2) 
$$
is an isomorphism, where $X=\Sigma\times_{\ZZ^2} \RR^2$. 

By the measured foliated twisted L$^2$-index theorem (see Section \ref{FoliatedIndex}),
$$
\tau_\mu(\mu_\theta(\xi)) = \int_{X} e^{\theta dx_1\wedge dx_2} \wedge {\rm Ch}(\xi) 
$$
$X$ is connected since the $\ZZ^2$-action is minimal cf. Lemma 3 \cite{BO-JFA}, so every vector bundle $\xi$ on $X$ has constant rank.
\begin{align*}
& =  \theta\int_{\TT^2} dx_1 \wedge dx_2 \, \mu(\Sigma) \text{rank}(\xi) +  \int_{X} c_1(\xi)\\
& = \theta\, \text{rank}(\xi) +  \int_{X} c_1(\xi)
\end{align*}
Varying over all vector bundles $\xi$ on $X$, and using Bellissard's gap-labelling theorem in 2D \cite{BCL} (when the magnetic field vanishes i.e. $\theta=0$), the result follows.
\end{proof}

\begin{remark} In section 5, \cite{FH}, they compute a useful example which we now recall and that does not use their Theorem 4.2 in \cite{FH}.
Suppose that $0<\alpha_1<\alpha_2<1$ are two rationally independent irrational numbers. Then $T_j x = x + \alpha_j \, (\text{mod}\, 1), \, j=1, 2$  defines a minimal $\ZZ^2$-action on the circle $\RR/\ZZ$. Define the Cantor set $\Sigma$ to be the circle disconnected along the orbit of $\ZZ^2$ through the origin. Then by fiat, $\ZZ^2$ also acts minimally on $\Sigma$ and this example has a unique invariant probability measure $\mu$. 
What is shown on page 623 in \cite{FH} is that in this case, one has,
$$
 \int_{X} c_1(\xi) \in  \ZZ[\mu] = \ZZ + \ZZ \alpha_1 + \ZZ \alpha_2.
$$
Therefore the magnetic gap-labelling theorem, Corollary \ref{2Dgapthm}, in this particular 2D case is:
$$
\tau^\mu(K_0(C(\Sigma)\rtimes_\sigma\Z^2)) =  \ZZ + \ZZ \alpha_1 + \ZZ \alpha_2 + \Z \theta. 
$$

\end{remark}

\begin{remark}
Corollary \ref{2Dgapthm} was known to Kellendonk, in \cite{Kell}, and we thank him for pointing it out to us. He uses the bulk-boundary correspondence method to reduce to the 1D case. Recently another proof using groupoids has been given in \cite{Kreisel}.
\end{remark}

\begin{remark}
In \cite{Raikov}, they consider the 2D magnetic Schr\"odinger operator,
$$
    H = -\frac{\partial^2}{\partial x^2} +
    \left(-i\frac{\partial}{\partial y} - \theta x \right)^2 + V(x).
    $$
Here $B=\theta\, dx \wedge dy, \, \theta \ne 0$ is a constant magnetic field, and $V$
is a non-constant smooth $\tau$-periodic electric potential that is independent of the $y$ variable.
The self-adjoint operator $H$ on $L^2(\RR^2)$ is proved in \cite{Raikov} to generically have {\em infinitely} many open spectral gaps. 
This is in stark contrast to the Bethe-Sommerfeld conjecture, 
which says that there are only a {\em finite} number of gaps in the spectrum 
of any Schr\"odinger operator with smooth periodic potential $V$ on Euclidean space, 
in the case when the magnetic field vanishes, i.e. $\theta =0$, 
whenever the dimension is greater than or equal to $2$.
The Bethe-Sommerfeld conjecture was proved by L. Parnovski in \cite{Parnovski}. 
In fact in \cite{Raikov}, they also study the Hamiltonian $H_\pm = H \pm W$, where $W\in L^\infty(\RR^2)$ is non-negative and decays at infinity and $\theta\ne 0$,  so that $H_\pm$ is the sort of Hamiltonians that we consider in our paper. They 
find in  \cite{Raikov} that
there are {\em infinitely} many discrete eigenvalues of $H_\pm$ in any open gap in
the spectrum of $\text{spec}(H)$, and the convergence of
these eigenvalues to the corresponding endpoint of the spectral gap is asymptotically Gaussian. This shows that the spectral gaps of magnetic Schr\"odinger operators (of the type considered in this paper) can be rather interesting even in higher dimensions. 

In fact, let $\Theta$ be a skew-symmetric $(2n \times 2n)$ matrix. Putting $\Theta$ in Jordan normal form, we can assume without loss of generality that the associated magnetic field $B=\frac{1}{2} dx^t \Theta dx = \sum_{i=1}^n \Theta_{2i-1, 2i} dx_{2i-1} \wedge dx_{2i}$. Choosing the vector potential $A= \sum_{i=1}^n \Theta_{2i-1, 2i} x_{2i-1} \wedge dx_{2i}$, we see that with $H_A = (d+iA)^\dagger(d+iA) $ that $H_{A,V}  = H_A + V$, where
$$
H_{A,V} = -  \sum_{i=1}^n\left( \left(\frac{\partial}{\partial x_{2i-1}}\right)^2 
+   \sum_{i=1}^n \left(-i \frac{\partial}{\partial x_{2i}} -  \Theta_{2i-1, 2i} \, x_{2i-1}\right)^2 + V_i( x_{2i-1}) \right)
$$
Arguing exactly as in the 2D case, we see that $H_A$ has discrete spectrum with infinite multiplicity, where $V_i$ is a smooth periodic function similar to the 2D case. The argument of \cite{Raikov} easily extends to the higher (even) dimensional case in this way, giving examples  of magnetic Schr\"odinger operators (of the type considered in this paper) that have infinitely many open spectral gaps that are interesting.\\
\end{remark}


\section{Proof of the conjecture in the Jordan block diagonal case}


In this section, we establish the magnetic gap labelling conjecture  \ref{mainconj} in the case when the skew symmetric matrix associated to the constant magnetic field on Euclidean space is in Jordan block diagonal form. Although we consider only even dimensional systems here, the odd dimensional case is similar.

Given a constant magnetic field on Euclidean space $\RR^2$ determined by $\theta \in \R$,  and a $\Z^2$-invariant probability measure $\mu$ on the disorder set $\Sigma$ that is a Cantor set, where the action is minimal,  we have seen in Corollary \ref{2Dgapthm} that the range of the trace on $K$-theory is,
$$
\tau^\mu(K_0(C(\Sigma) \rtimes_\theta \ZZ^2)) = \Z[\mu] + \theta\Z.
$$

Now consider the special 4D situation of the same phenomenon, where the skew-symmetric matrix $\Theta$ is in Jordan block diagonal form,
$$
\Theta = \left(\begin{array}{cccc}0 & -\theta_1 & 0 & 0 \\ \theta_1 & 0 & 0 & 0 \\0 & 0 & 0 & -\theta_2 \\0 & 0 & \theta_2 & 0\end{array}\right)
$$
Consider $\Z^2$-invariant probability measures $\mu_i$ on the disorder sets $\Sigma_i$ that are Cantor sets for $i=1,2$, where the actions are minimal.

Let $\Sigma = \Sigma_1 \times \Sigma_2$. Then $\Z^4$ acts minimally on $\Sigma$ with invariant measure $\mu=\mu_1\times\mu_2$
and one can form the twisted crossed product algebra,
$$
C(\Sigma) \rtimes_\Theta \Z^4 \cong (C(\Sigma_1) \rtimes_{\theta_1} \ZZ^2) \otimes (C(\Sigma_2) \rtimes_{\theta_2} \ZZ^2)
$$
Recall that  the Kunneth Theorem for Tensor Products \cite{Schochet1982} asserts that if  $A$ and $B$ are $C^*$-algebras, with $A$ being nuclear
and $K_\bullet(A)$ being torsion-free, 
then there is a natural isomorphism,
$$
K_0(A\otimes B) \cong K_0(A) \otimes K_0(B) \oplus K_1(A) \otimes K_1(B)
$$
Note that $A= C(\Sigma_1) \rtimes_{\theta_1} \ZZ^2$ and $B= C(\Sigma_2) \rtimes_{\theta_2} \ZZ^2$ satisfy the hypotheses of the Kunneth theorem,
and that 
$$
K_1(C(\Sigma) \rtimes_\theta \ZZ^2)) = \Z[u_1] + \Z[u_2],
$$
where $u_1, u_2$ are the unitaries generating $K_1(C^*(\Z^2))$. Then the tensor product of the $K_1$ groups is,
$$
(\Z[u_1] \oplus \Z[u_2])(\Z[u_1] \oplus \Z[u_2])= \Z[u_1 \cup u_1] \oplus \Z[u_2\cup u_2] + \Z[u_1 \cup u_2]
$$
Applying the trace, we see that 
$$
\tau^\mu(K_1(C(\Sigma) \rtimes_\theta \ZZ^2)) \otimes K_1(C(\Sigma) \rtimes_\theta \ZZ^2)) ) \subset \Z.
$$

Therefore the range of the trace on $K$-theory is,
\begin{align*}
\tau^\mu(K_0(C(\Sigma) \rtimes_\Theta \Z^4)) & = (\Z[\mu_1] + \theta_1\Z) (\Z[\mu_2] + \theta_2\Z)\\
&= \Z[\mu_1] \Z[\mu_2] + \theta_2 \Z[\mu_1]  +  \theta_1 \Z[\mu_2] + \theta_1\theta_2\Z\\
& \subset \Z[\mu_1 \times \mu_2] + \theta_2 \Z[\mu_1]  +  \theta_1 \Z[\mu_2] + \theta_1\theta_2\Z
\end{align*}
since it is not hard to see that $ \Z[\mu_1] \Z[\mu_2]  \subset \Z[\mu_1 \times \mu_2] $. This proves another special case of our magnetic gap-labelling conjecture \ref{mainconj}
and also motivates our conjecture.

By elementary induction, this works for all even dimensional systems of this sort. More precisely, consider the constant magnetic field on 
Euclidean space $\R^{2n}$  given by the skew-symmetric matrix in Jordan diagonal form, 
$$
\Theta = \bigoplus_{j=1}^n \left(\begin{array}{cc}0 & -\theta_j \\ \theta_j & 0\end{array}\right)
$$
Consider $\Z^2$-invariant probability measures $\mu_i$ on the disorder sets $\Sigma_i$ that are Cantor sets for $i=1,2,\ldots, n$, where the actions are all minimal. Let $\Sigma= \Sigma_1 \times \Sigma_2 \ldots \times \Sigma_n$. Then $\Z^{2n}$ acts on $\Sigma$ and one can form the twisted crossed product algebra,
$$
C(\Sigma) \rtimes_\Theta \Z^{2n} \cong \bigotimes_{j=1}^n (C(\Sigma_j) \rtimes_{\theta_j} \ZZ^2) 
$$
Assume the induction hypothesis that 
$$
\tau^\mu(K_0(C(\Sigma) \rtimes_\Theta \Z^{2n})) \subset \sum_I \theta_I \Z[\mu_{I^c}] 
$$
where $\theta_I = \prod_{j\in I} \theta_j$, $\quad \mu_{J} = \prod_{j\in J} \mu_j$, $\quad\mu = \prod_{j=1}^n \mu_j$,  $I^c$ denotes the index set that is complement to $I$ and $|I|\le n$. This is exactly the statement of our conjecture in 
this special case. 

Now let $\Sigma' = \Sigma \times \Sigma_{n+1}$, with $\Z^2$-invariant probability measure $\mu_{n+1}$ on the disorder set $\Sigma_{n+1}$ that is a Cantor set, where the actions is minimal. Consider the constant magnetic field on 
Euclidean space $\R^{2n+2}$  given by the skew-symmetric matrix in Jordan diagonal form, 
$$
\Theta' = \Theta \oplus \left(\begin{array}{cc}0 & -\theta_{n+1} \\ \theta_{n+1} & 0\end{array}\right)
$$
Then $\Z^{2n+2}$ acts on $\Sigma'$ and one can form the twisted crossed product algebra,
$$
C(\Sigma') \rtimes_{\Theta'} \Z^{2n+2} \cong C(\Sigma) \rtimes_\Theta \Z^{2n} \otimes  C(\Sigma_{n+1}) \rtimes_{\theta_{n+1}} \ZZ^2.
$$
Using the Kunneth Theorem for Tensor Products \cite{Schochet1982} and arguing as before, we see that the range of the trace on $K$-theory is,
\begin{align*}
\tau^\mu(K_0(C(\Sigma') \rtimes_{\Theta'} \Z^{2n+2})) & \subset (\sum_I \theta_I \Z[\mu_{I^c}] ) (\Z[\mu_{n+1}] + \theta_{n+1}\Z)\\
& \subset   \sum_I \left(\theta_I \Z[\mu_{I^c \cup (n+1)}]  +\sum_I \theta_{I\cup (n+1)} \Z[\mu_{I^c}]  \right),
\end{align*}
completing the induction step for $n+1$, thereby establishing our magnetic gap-labelling conjecture \ref{mainconj} in the Jordan block diagonal case.\\


\section{The 3D case}\label{sect:3D}


{{We restrict ourselves in this section to the 3D case where we have succeeded to prove Conjecture \ref{mainconj} and Conjecture
\ref{mainconj-minimal}.}}

\subsection{Proof of Conjecture \ref{mainconj}}\label{3D}

We now proceed to prove Conjecture \ref{mainconj} in the 3D case. We prove more precisely that the magnetic gap-labelling group of an aperiodic tiling   which corresponds to an action of $\Z^3$ on the Cantor space $\Sigma$, is contained in the magnetic frequency group defined in Section \ref{Conjecture}. We assume  that there are no nontrivial  globally invariant open subspaces in $\Sigma$. The Boolean algebra  of clopen subspaces of $\Sigma$ is denoted $\maS$ and  it is also endowed with the induced action of $\Z^3$. The generators of the $\Z^3$ action on $\Sigma$ but also the induced one on $\maS$ are denoted generically $(T_i)_{1\leq i\leq 3}$. The free subgroup of $\Z^3$ generated by $T_j$ is again denoted $\langle T_j \rangle $. 

For any $1\leq i < j \leq 3$, we denote as before by $\Z_{ij}[\mu]$  the subgroup of the real line  generated by $\mu$-integrals of $\Z$-valued functions on  $\Sigma$ whose image in the coinvariants under $\Z^{\{i, j\}^c}$ is $\Z^{\{i, j\}}$-invariant.  
A multiplier $\sigma\in H^2(\Z^3, \R/\Z) \simeq \Lambda^2(\R/\Z)^3$ as before  is associated with a skew matrix $\Theta\in M_3(\R)$ and hence with the three real numbers $\Theta_{12}, \Theta_{13}$ and $\Theta_{23}$, which are the entries of the matrix $\Theta$.  We are now in position to state the main result of this section.

\begin{theorem}\label{Gap3d}\
With the previous notations, 
$$
 \Range (\tau^\mu_*) \; \subset \; \Z [\mu] \; +\; \Theta_{12}\; \Z_{12} [\mu] \;+ \; \Theta_{13}\; \Z_{13}[\mu] \; + \; \Theta_{23}\; \Z_{23}[\mu]. 
$$
\end{theorem}

{{Before embarking on the long proof of Theorem \ref{Gap3d}, 
we first outline the strategy and the main steps of this proof for the convenience of the reader. 
Lemma \ref{lem:gap1} simplifies the topological side of the measured twisted foliated index theorem (see subsection \ref{FoliatedIndex})
in the 3D case, using in particular the Baum-Connes map \cite{BCH}.  Using Lemma \ref{lem:gap1}  and some homological algebra of group cohomology with coefficients in a 
module, Corollary \ref{cor:MGL+homological algebra} identifies the magnetic gap labelling group in these terms. {{Further homological properties are  exploited in both Lemma \ref{alpha} and Lemma \ref{beta}, together with integrality of the Chern character in the 3D case, to reduce the proof of Theorem \ref{Gap3d} to the computation of the range of the {\em{integral}} group cohomology with coefficients, under the cup-product morphism with respect to the  magnetic field. The computation of this latter range yields to
 the proof of Theorem \ref{Gap3d} and verifies Conjecture \ref{mainconj} in the 3D case. In subsection \ref{sect:3Dstrong minimality}, upon assuming that the action 
 is strongly minimal,
 Lemma \ref{BoundaryMap}  studies the boundary map in group cohomology with coefficients in a 
module and reduces the proof of Conjecture \ref{mainconj-minimal}  to an injectivity condition, see Theorem \ref{3dConj2}. Finally the injectivity condition of Theorem \ref{3dConj2}  is established for strongly minimal actions in 
Corollary \ref{3Dconj2}, proving Conjecture \ref{mainconj-minimal} in the 3D case.}}}

\medskip
{{The rest of the section will thus be devoted to this long direct computation that we have split into lemmas for the sake of clarity as described above.}}  Recall that the mapping torus is the space $X:=(\Sigma\times \R^3)/\Z^3$ where we have divided out by the diagonal action. This is a transversely Cantor foliated space which fibres  over the three torus $\T^3$, in particular pulling back cohomology classes on $\T^3$ we get cohomology classes on $X$. There is a well defined ``Poincar\'e duality'' isomorphism which was explicitly described in \cite{BO-JFA} using leafwise cohomology:
$$
\Psi_{\Z^3} : H^3 (X, \R) \longrightarrow C(\Sigma, \R)_{\Z^3}  \text{ and also the Cech version } \Psi_{\Z^3} : H^3 (X, \Z) \longrightarrow C(\Sigma, \Z)_{\Z^3} .
$$
Here and as in the other sections, for any subgroup $\Gamma$, the subscript $(\; )_\Gamma$ refers to coinvariants  while the superscript $(\; )^\Gamma$ refers to  invariants. Recall the $2$-cohomology class $B$ on the torus $\T^3$ which is associated with $\sigma$. 

\begin{lemma}\label{lem:gap1}
The magnetic gap-labelling group is given by
$$
 \Range (\tau^\mu_*) \; = \; \Z[\mu] \; +\;  \, \left\{  \left\langle \mu\circ  \Psi_{\Z^3},  \tr \left(\frac{u^{-1}\, du}{2i\pi}\right)\, \cup \,B \right\rangle , u\in K^1(X)  \right\}.
$$
\end{lemma}

\begin{proof}
Denote by $\dirac_\sigma$ the leafwise $\sigma$-twisted Dirac operator (see subsection \ref{FoliatedIndex}) along the leaves of our foliated space $X$.  By classical arguments about the Baum-Connes map for $\Z^3$ with coefficients in the $\Z^3$-algebra $C(\Sigma)$, it is easy to check that the twisted Connes' Thom isomorphism of Theorem \ref{twistedCT},
$K^1(X) \rightarrow K_0(C(X)\rtimes_\sigma \R^3)$  coincides with the twisted foliated index map, see subsection \ref{FoliatedIndex} and also \cite{BO-JFA}. More precisely, the $K$-theory of $X$ can be trivially identified with the $K$-theory of the algebra $C^{\infty, 0} (X)$  of continuous leafwise smooth functions on $X$. Given a unitary $u$ in a matrix algebra of $C^{\infty, 0}(X)$, we may consider the Toeplitz operator $T_{\sigma, u}$ associated with $\dirac_\sigma$ and with symbol $u$.  This is a leafwise elliptic $0$-th order pseudodifferential operator on the foliated space $X$ and it has a well defined index class $\Ind (T_{\sigma, u})$ in $K_0( C^{\infty, 0} (X)\rtimes_\sigma \R^3)$ which is defined in \cite{BenameurMathaiPreprint} following the lines of \cite{MooreSchochet}. So the Connes' Thom isomorphism is described as the map $
[u] \longmapsto \Ind (T_{\sigma, u}).$  On the twisted foliation algebra $ C^{\infty, 0} (X)\rtimes_\sigma \R^3$ we have the semi-finite normal trace, denoted equally $\tau^\mu$, associated with the monodromy invariant measure defined by $\mu$, and we know that this trace and the original trace $\tau^\mu$ on $C(\Sigma)\rtimes_\sigma \Z^3$ agree in $K$-theory with respect to the Morita isomorphism (see \cite{BenameurPiazza})
$$
K_0(C(X)\rtimes_{\sigma} \R^3)\longrightarrow K_0(C(\Sigma)\rtimes_\sigma \Z^3).
$$
As a corollary of this discussion, we see that the magnetic gap-labelling group coincides with the range of the map
$$
K^1 (X) \longrightarrow \R \text{ given by } [u] \longmapsto \tau^\mu_*\left(\Ind (T_{\sigma, u})\right).
$$
We now apply the twisted foliated index theorem from subsection \ref{FoliatedIndex} and \cite{BenameurMathaiPreprint} which gives exactly the statement of the lemma. Recall indeed that $\ch (u)=\ch_1(u)+\ch_3(u)$ where $\ch_1(u)=\tr \left(\frac{u^{-1}\, du}{2i\pi}\right)$. On the other hand  the  gap-labelling theorem for $p=3$ implies that the group 
$$
\left\{ \left\langle \mu, \int_{(0,1)^3} \ch_3 (u) \right\rangle , u\in K^1(X) \right\}\, \text{ coincides with }\Z[\mu]. 
$$ 
\end{proof}

It remains thus to identify the second additive subgroup appearing in the previous lemma. Notice that the range of $\ch_1: K^1(X) \rightarrow H^1 (X, \Q)$ is exactly given by $H^1(X, \Z)$ and is hence  isomorphic through $\Psi_{\Z^3}$ to $H^1(\Z^3, C(\Sigma, \Z))$.  Moreover, the following diagram commutes:
$$
\begin{CD}
H^* (X, \Z) \otimes H^*(\T^3, \R)@>{\cup}>> H^*(X, \R) \\
@V{\nu \otimes P }VV     @VV{\nu_\R}V\\
H^*(\Z^3, C(\Sigma, \Z)) \otimes H^*(\Z^3, \R)  @>\cup>> H^*(\Z^3, C(\Sigma, \R))
\end{CD}
$$
where $\nu$ is the isomorphism $H^* (X, \Z) \rightarrow H^*(\Z^3, C(\Sigma, \Z))$ and $\nu_\R$ its version with real coefficients. Here $P$ is the Pontryagin isomorphism $H^*(\T^3, \R)\rightarrow H^*(\Z^3, \R)$.
\medskip

We  denote as before  by $(\psi_j)_{1\leq j\leq 3} $ the generators of $H^1(\Z^3, \Z)$. So, $\Theta$ corresponds to the element of $H^2(\Z^3, \R)$, still denoted by $\Theta$, given by
$$
\Theta := \Theta_{12} \psi_1\cup \psi_2 + \Theta_{13} \psi_1\cup \psi_3 + \Theta_{23} \psi_2\cup \psi_3,
$$
whose cohomology class  corresponds through the Pontryagin duality  to the de Rham cohomology class of the  form $B$ on $\T^3$. 
In view of the previous lemma, we need to compute the range of the composite  map
$$
H^1(X, \Z) \stackrel{\cup B}{\longrightarrow} H^3 (X, \R) \stackrel{\Psi_{\Z^3}}{\longrightarrow} C(\Sigma, \R)_{\Z^3} \stackrel{\mu}{\longrightarrow} \R.
$$

\begin{corollary}\label{cor:MGL+homological algebra}
The magnetic gap-labelling group coincides with the subgroup of $\R$ which is $\Z[\mu]$ plus the range of the map
$$
H^1 (\Z^3, C(\Sigma, \Z)) \stackrel{\cup \Theta}{\longrightarrow} H^3(\Z^3, C(\Sigma, \RR)) \stackrel{\Psi}{\longrightarrow} C(\Sigma, \R)_{\Z^3} \stackrel{\mu}{\longrightarrow} \R.
$$
where $\Psi$ is the usual Poincar\'e duality isomorphism for the group $\Z^3$.
\end{corollary}

\begin{proof}
From the previous considerations, we deduce the following commutative diagram
$$
\begin{CD}
H^1 (X, \Z) @>{\cup B}>> H^3(X, \R)  @>{\Psi_{\Z^3}}>> C(\Sigma, \R)_{\Z^3} \\
@V{\nu }VV     @VV{\nu_\R}V @VV{=}V\\
H^1(\Z^3, C(\Sigma, \Z))   @>\cup \Theta >> H^3(\Z^3, C(\Sigma, \R)) @>>{\Psi}> C(\Sigma, \R)_{\Z^3} 
\end{CD}
$$
This completes the proof of the corollary, upon using Lemma \ref{lem:gap1}.
\end{proof}

In the sequel, we shall  for simplicity no more denote the isomorphism $\nu: H^*(X, \Z) \simeq H^*(\Z^3, C(\Sigma, \Z))$ and hence denote by $\Psi_{\Z^3}$ the Poincar\'e duality isomorphism 
$$
\Psi_{\Z^3} : H^3(\Z^3, C(\Sigma, \Z)) \longrightarrow H_0(\Z^3, C(\Sigma, \Z))=C(\Sigma, \Z)_{\Z^3}.
$$
This discussion allows to deduce that the magnetic gap-labelling group coincides with $\Z[\mu]$ plus the sum of the ranges of the three maps corresponding to $1\leq i < j \leq 3$ which are
$$
H^1 (\Z^3, C(\Sigma, \Z)) \stackrel{\cup (\Theta_{ij}  \psi_i\cup \psi_j)}{\longrightarrow} H^3(\Z^3, C(\Sigma, \R) \stackrel{\Psi_{\Z^3}}{\longrightarrow} C(\Sigma, \R)_{\Z^3} \stackrel{\mu}{\longrightarrow} \R.
$$
It thus suffices to give the proof for $\Theta_{12} \psi_1\cup \psi_2$ and the two other ranges will be obtained similarly.  Since $\Theta_{12}$ is constant, we only  need to deal with the element $\psi_1\cup \psi_2$ in $H^2(\Z^3, \Z)$. We shall denote by
$$
\Psi_{\langle T_i, T_j \rangle }: H^2(\langle T_i, T_j \rangle , C(\Sigma, \Z)) \longrightarrow C(\Sigma, \Z)_{\langle T_i, T_j \rangle },
$$
the similar isomorphism to $\Psi_{\Z^3}$ but corresponding to the Poincar\'e duality for the mapping torus associated with the $\Z^2$ action on $\Sigma$ corresponding to the generators $T_i$ and $T_j$. Notice that there is an induced action of $T_2$ on $H^*(\langle T_1, T_3 \rangle , C(\Sigma, \Z))$ and on $C(\Sigma, \Z)_{\langle T_1 T_3 \rangle }$ and that $\Psi_{\langle T_1, T_3 \rangle }$ is $\langle T_2 \rangle $-equivariant \cite{BO-JFA}. In particular, the invariants 
$$
H^2(\langle T_1, T_3 \rangle , C(\Sigma, \Z))^{\langle T_2 \rangle }
$$ 
are sent under the map $\Psi_{\langle T_1, T_3 \rangle }$ into the invariants $\left(C(\Sigma, \Z)_{\langle T_1, T_3 \rangle }\right)^{\langle T_2 \rangle }$. 

From Theorem 8 in \cite{BO-JFA}, it is easy to deduce the  following exact sequences
\begin{equation}\label{ExactSequence1}
0\to \left(C(\Sigma, \Z)^{\langle T_1, T_3 \rangle }\right)_{\langle T_2 \rangle } \stackrel{i}{\longrightarrow} H^1 (\Z^3, C(\Sigma, \Z)) \stackrel{\pi}{\longrightarrow }H^1(\langle T_1, T_3 \rangle , C(\Sigma, \Z))^{\langle T_2 \rangle } \to 0
\end{equation}
and
\begin{equation}\label{ExactSequence2}
0 \to \left(C(\Sigma, \Z)^{\langle T_3 \rangle }\right)_{\langle T_1 \rangle } \stackrel{i'}{\longrightarrow} H^1(\langle T_1, T_3 \rangle , C(\Sigma, \Z)) \stackrel{\pi'}{\longrightarrow} \left(C(\Sigma, \Z)_{\langle T_3 \rangle }\right)^{\langle T_1 \rangle } \to 0.
\end{equation}

For the convenience of the reader, let us briefly describe the maps appearing in these exact sequences. The maps $i$ and $i'$ are pull-back maps corresponding to projections onto $\langle T_2\rangle$ and $\langle T_1\rangle$ respectively, composed with inclusions of coefficients. For instance, in the first exact sequence (\ref{ExactSequence1}),  $\left(C(\Sigma, \Z)^{\langle T_1, T_3 \rangle }\right)_{\langle T_2 \rangle } $ is first identified with $H^1(\langle T_2\rangle, C(\Sigma, \Z)^{\langle T_1, T_3\rangle})$. Then $i$ is the composite map
$$
\left(C(\Sigma, \Z)^{\langle T_1, T_3 \rangle }\right)_{\langle T_2 \rangle } \simeq H^1(\langle T_2\rangle, C(\Sigma, \Z)^{\langle T_1, T_3\rangle})\stackrel{\pi_2^*}{\longrightarrow} H^1(\Z^3, C(\Sigma, \Z)^{\langle T_1, T_3\rangle}) \rightarrow H^1(\Z^3, C(\Sigma, \Z)).
$$
Here we have denoted  by $\pi_j: \Z^3\rightarrow \langle T_j\rangle$ the projection.
The similar description holds for the map $i'$ in the second exact sequence. 

In the same way, the maps $\pi$ and $\pi'$ are just  pull-back maps. More precisely, again in the first exact sequence, the natural inclusion  $\iota: \langle T_1, T_3\rangle \hookrightarrow \Z^3$ obtained by crossing with zero for the missing $\langle T_2\rangle$, induces 
$$
\iota^* : H^1 (\Z^3, C(\Sigma, \Z)) \longrightarrow H^1 (\langle T_1, T_3\rangle, C(\Sigma, \Z)),
$$
and it is easy to see that the range of this map is exactly $H^1 (\langle T_1, T_3\rangle, C(\Sigma, \Z))^{\langle T_2\rangle}$. For instance, if $\iota_2: \langle T_2\rangle \hookrightarrow \Z^3$ is the similar inclusion then for any $g\in \langle T_1, T_3\rangle$ and any $g_2\in \langle T_2\rangle$, one checks the following relation  for any $1$-cocycle $c\in H^1 (\Z^3, C(\Sigma, \Z))$
$$
(g_2 \iota^*c - c ) (g) = (g \iota_2^*c - \iota_2^*c) (g_2).
$$
Again the map $\pi'$ is defined similarly.

It is then clear by compatibility of cup products with pull-backs that  the map
$$
\cup \psi_2: H^1 (\Z^3, C(\Sigma, \Z)) \longrightarrow H^2 (\Z^3, C(\Sigma, \Z)),
$$
vanishes on the image of $i$, say on the  subgroup $H^1(\langle T_2 \rangle, C(\Sigma, \Z)^{\langle T_1, T_3 \rangle})$, and hence we deduce that cup product with $\psi_1\cup \psi_2$ induces a well defined map 
$$
\alpha: H^1(\langle T_1, T_3 \rangle , C(\Sigma, \Z))^{\langle T_2 \rangle } \longrightarrow H^3 (\Z^3, C(\Sigma, \Z)).
$$

\begin{lemma}\label{alpha}
The composite map $\Psi_{\Z^3} \circ \alpha$ coincides, up to sign,  with the expected map
\begin{multline*}
H^1(\langle T_1, T_3 \rangle , C(\Sigma, \Z))^{\langle T_2 \rangle } \stackrel{\cup \psi_1}{\longrightarrow} H^2 (\langle T_1, T_3 \rangle , C(\Sigma, \Z))^{\langle T_2 \rangle } \\
\stackrel{\Psi_{\langle T_1, T_3 \rangle }}{\longrightarrow} \left(C(\Sigma, \Z)_{\langle T_1, T_3 \rangle }\right)^{\langle T_2 \rangle } 
\longrightarrow C(\Sigma, \Z)_{\Z^3}.
\end{multline*}
\end{lemma}

\begin{proof}
The proposed composite map is denoted $\Psi_{\Z^3} \circ \alpha'$ and it is clearly well defined. The map $\alpha$ is also well defined and it is by definition induced by
$$
\cup (\psi_1\cup \psi_2): H^1 (\Z^3, C(\Sigma, \Z)) \longrightarrow H^3 (\Z^3, C(\Sigma, \Z)).
$$
Since the map $H^1 (\Z^3, C(\Sigma, \Z))\rightarrow H^1(\langle T_1, T_3 \rangle , C(\Sigma, \Z))^{\langle T_2 \rangle }$ of \eqref{ExactSequence1}  is an epimorphism, it remain to show that $\alpha'$ fits in a commutative diagram
$$
\begin{CD}
H^1 (\Z^3, C(\Sigma, \Z)) @>>> H^1(\langle T_1, T_3 \rangle , C(\Sigma, \Z))^{\langle T_2 \rangle }  \\
@V{\cup (\psi_1\cup\psi_2)}VV     @VV{\alpha'} V\\
H^3 (\Z^3, C(\Sigma, \Z)) @> {=} >> H^3 (\Z^3, C(\Sigma, \Z))
\end{CD}
$$
As recalled above, the epimorphism $H^* (\Z^3, C(\Sigma, \Z)) \longrightarrow H^*(\langle T_1, T_3 \rangle , C(\Sigma, \Z))^{\langle T_2 \rangle }$ is given by restriction using the inclusion $\iota$ of the subgroup  $\langle T_1, T_3 \rangle $ in $\Z^3$, and using that the restricted  cocycles are automatically $\langle T_2 \rangle $ invariant.  Therefore, compatibility of cup products with pullbacks yields the commutativity of the following diagram
$$
\begin{CD}
H^1 (\Z^3, C(\Sigma, \Z)) @>{\iota^*}>> H^1(\langle T_1, T_3 \rangle , C(\Sigma, \Z))^{\langle T_2 \rangle }  \\
@V{\cup \psi_1}VV     @VV{\cup\psi_1} V\\
H^2 (\Z^3, C(\Sigma, \Z)) @> {\iota^*} >> H^2 (\langle T_1, T_3 \rangle, C(\Sigma, \Z))^{\langle T_2 \rangle }
\end{CD}
$$
Now, if $[\Z^3]\in H_3(\Z^3, \Z)$ is the fundamental class which embodies, through cap product, the Poincar\'e duality map $\Psi_{\Z^3}$, and if similarly $[\langle T_1, T_3\rangle]\in H_2(\langle T_1, T_3\rangle, \Z)$ is the corresponding fundamental class for the subgroup $\langle T_1, T_3\rangle$, then the following relation holds for any $c\in H^2(\Z^3, C(\Sigma, \Z))$:
$$
(c \cup \psi_2)\cap [\Z_3] = \pm J_2\left(\iota^*c \cap [\langle T_1, T_3\rangle]\right) \quad \in C(\Sigma, \Z)_{\Z^3},
$$
where $J_2: C(\Sigma, \Z)_{\langle T_1, T_3\rangle} \rightarrow C(\Sigma, \Z)_{\Z^3}$. The reason this relation holds is simply that $\psi_2\cap [\Z^3]$ is the $2$-homology class which is dual to $\pm \psi_1\cup \psi_3$. The proof is now complete.
\end{proof}

We denote in the following lemma by $A$ the image of $H^1(\langle T_1, T_3 \rangle , C(\Sigma, \Z))^{\langle T_2 \rangle }$  in $\left(C(\Sigma, \Z)_{\langle T_3 \rangle }\right)^{\langle T_1 \rangle }$ under the epimorphism of the above second exact sequence \eqref{ExactSequence2}.

\begin{lemma}\label{beta}
The map $\alpha$ induces a well defined morphism $\beta: A \rightarrow H^3(\Z^3, C(\Sigma, \Z))$ such that the composite map $\Psi_{\Z^3}\circ \beta$ is given by the natural map
$$
 \left( C(\Sigma, \Z)_{\langle T_3 \rangle }\right)^{\langle T_1 \rangle } \longrightarrow C(\Sigma, \Z)_{\Z^3}.
$$
\end{lemma}

\begin{proof}
From the previous lemma \ref{alpha}, we see that $\alpha$ is a composite map $\theta\circ (\cup \psi_1)$ with $\theta$ some morphism. Clearly, the map $\cup\psi_1$ vanishes on $H^1 (\langle T_1 \rangle , C(\Sigma, \Z)^{\langle T_3 \rangle })$, hence $\alpha$ vanishes on the kernel of the epimorphism $H^1(\langle T_1, T_3 \rangle , C(\Sigma, \Z)) \longrightarrow H^1(\langle T_3 \rangle , C(\Sigma, \Z))^{\langle T_1 \rangle }$ and finally also on its invariants under the group $\langle T_2 \rangle $. We deduce that the morphism $\beta$ is well defined. 

The proposed composite map can be written as $\Psi_{\Z^3}\circ \beta'$ and it is clearly well defined and can then be restricted to the subgroup $A$. Now, the restriction of $\pi'$ yields  the epimorphism
$$
H^1(\langle T_1, T_3 \rangle , C(\Sigma, \Z))^{\langle T_2 \rangle } \stackrel{\pi'}{\longrightarrow} A.
$$
Hence, it remains as in the proof of the previous lemma to show that $\beta'$ fits in the following commutative diagram
$$
\begin{CD}
H^1(\langle T_1, T_3 \rangle , C(\Sigma, \Z)) @>{{\iota'}^*}>> \left(H^1(\langle T_3\rangle, C(\Sigma, \Z)\right)^{\langle T_1 \rangle }  \\
@V{\cup \psi_1}VV     @VV{J_1\circ \Psi_{\langle T_3\rangle}}V\\
H^2 (\langle T_1, T_3 \rangle , C(\Sigma, \Z)) @>{\Psi_{\langle T_1, T_3 \rangle }}>> C(\Sigma, \Z)_{\langle T_1, T_3 \rangle }
\end{CD}
$$
where $\iota': \langle T_3\rangle \hookrightarrow  \langle T_1, T_3\rangle$ is the inclusion as before and $J_1: C(\Sigma, \Z)_{\langle T_3\rangle} \rightarrow C(\Sigma, \Z)_{\langle T_1, T_3\rangle}$ is again the natural quotient map. The Poincar\'e maps $\Psi_{\langle T_3\rangle}$ and $\Psi_{\langle T_1, T_3 \rangle }$ are cap products by fundamental classes which are denoted respectively $[\langle T_3\rangle]$ and $[\langle T_1, T_3 \rangle]$ and the homology class $\psi_1\cap [\langle T_1, T_3 \rangle]$ clearly coincides with the Poincar\'e dual to $\psi_3$. Hence, for any $c\in H^1(\langle T_1, T_3 \rangle , C(\Sigma, \Z))$, we can write
$$
(c\cup \psi_1) \cap  [\langle T_1, T_3 \rangle] = \pm J_1\left( {\iota'}^*c \cap [\langle T_3\rangle]  \right).
$$
Therefore, the proof is complete. 
\end{proof}

\begin{proof} (of Theorem \ref{Gap3d})
 
From Lemma \ref{alpha}, we deduce that the range of the map 
$$
H^1 (\Z^3, C(\Sigma, \Z)) \stackrel{\cup (\psi_1\cup \psi_2)}{\longrightarrow} H^3(\Z^3, C(\Sigma, \Z) \stackrel{\Psi_{\Z^3}}{\longrightarrow} C(\Sigma, \Z)_{\Z^3} \stackrel{\mu}{\longrightarrow} \R.
$$
coincides with  the range of the map
$$
\langle \mu, \bullet\rangle  \circ \Psi_{\langle T_1, T_3 \rangle }\circ (\cup \psi_1): H^1( \langle T_1, T_3 \rangle , C(\Sigma, \Z))^{\langle T_2 \rangle } \longrightarrow \R. 
$$
From Lemma \ref{beta}, we further deduce that the range of this latter map  is equal to the range of $A$ under the map $\langle \mu, \bullet\rangle  \circ \Psi_{\langle T_3 \rangle }$. Since $A$ is contained in $H^1(\langle T_3 \rangle , C(\Sigma, \Z))^{\langle T_1, T_2 \rangle }$ we deduce from the $\langle T_2 \rangle $-equivariance of   $\Psi_{\langle T_3 \rangle }$ that the magnetic gap-labelling group is contained in $\Z[\mu]$ plus the image under $\langle \mu, \bullet\rangle $ of $\left(C(\Sigma, \Z)_{\langle T_3 \rangle }\right)^{\langle T_1, T_2 \rangle }$. Since this latter is by definition $\Z_{12}[\mu]$ the computation is complete for the pairing with $\psi_1\cup\psi_2$.

Reproducing the same proof for $\psi_1\cup \psi_3$ and $\psi_2\cup \psi_3$ respectively, we deduce that the corresponding ranges are contained respectively in the range under $\mu$ of 
$$
\left(C(\Sigma, \Z)_{\langle T_2 \rangle }\right)^{\langle T_1, T_3 \rangle }\; \text{ and }\; \left(C(\Sigma, \Z)_{\langle T_1 \rangle }\right)^{\langle T_2, T_3 \rangle },
$$
that is in $\Z_{13}[\mu]$ and $\Z_{23}[\mu]$.
If we sum up using again that $\Theta$ is constant, we see that the proof is complete.
\end{proof}

\subsection{Proof of Conjecture \ref{mainconj-minimal}}\label{sect:3Dstrong minimality}

{{We now prove Conjecture \ref{mainconj-minimal} in the 3D case. So, we assume that $T_1$, $T_2$ and $T_3$ all act minimally, this is the strong minimality condition. Let us show more precisely that if the action of the subgroup $\langle T_3\rangle$ is minimal then $\Z_{12}[\mu] $ is contained in our magnetic gap-labelling group. Here and as before we have denoted by   $\Z_{12}[\mu] = \mu ([C(\Sigma, \Z)_{\langle T_3\rangle}]^{\langle T_1, T_2\rangle}).$ So, this result will only use the condition that $T_3$ acts minimally and in fact a priori a weaker assumption, see Theorem \ref{3dConj2} below. Then the same statement can be proved for $T_1$ and $T_2$ yielding to our proof of Conjecture \ref{mainconj-minimal}. }}

{{We notice that the exact sequence \eqref{ExactSequence2} is $\langle T_2\rangle$-equivariant and we thus deduce the cohomology long sequence}}
\begin{multline*}
{{0 \to H^0 (<T_2>, [C(\Sigma, \Z)^{\langle T_3\rangle}]_{\langle T_1\rangle}) \longrightarrow H^0 (<T_2>, H^1(<T_1, T_3>, C(\Sigma, \Z))) \longrightarrow}} \\  {{H^0 (<T_2>, [C(\Sigma, \Z)_{\langle T_3\rangle}]^{\langle T_1\rangle}) }} {{\stackrel{\partial}{\longrightarrow} H^1 (<T_2>, [C(\Sigma, \Z)^{\langle T_3\rangle}]_{\langle T_1\rangle})\simeq  [C(\Sigma, \Z)^{\langle T_3\rangle}]_{{\langle T_1, T_2\rangle}} \longrightarrow \cdots}}
\end{multline*}
{{We thus need to describe the  boundary map $\partial$, and more precisely the map
$$
\hat\partial:  \left( H^1 (\langle T_3\rangle, C(\Sigma, \Z)\right)^{\langle T_1, T_2\rangle} \longrightarrow [C(\Sigma, \Z)^{\langle T_3\rangle}]_{{\langle T_1, T_2\rangle}},
$$
obtained out of $\partial$ using the following two isomorphisms 
$$
\left( H^1 (\langle T_3\rangle, C(\Sigma, \Z)\right)^{\langle T_1, T_2\rangle} \simeq H^0 (<T_2>, [C(\Sigma, \Z)_{\langle T_3\rangle}]^{\langle T_1\rangle})
$$ 
and
$$
H^1 (<T_2>, [C(\Sigma, \Z)^{\langle T_3\rangle}]_{\langle T_1\rangle})\simeq  [C(\Sigma, \Z)^{\langle T_3\rangle}]_{{\langle T_1, T_2\rangle}}.
$$ 
We can state the following general result for our 3D dynamical systems:}}

\begin{lemma}\label{BoundaryMap}
\begin{enumerate}
\item An element of  $\left( H^1 (\langle T_3\rangle, C(\Sigma, \Z)\right)^{\langle T_1, T_2\rangle}$ is a class $[\psi]$ of a $1$-cocycle $\psi\in Z^1 (\langle T_3\rangle, C(\Sigma, \Z)$ such that there exists $f, f'\in C(\Sigma, \Z)$ with
$$
T_1 \psi (1) - \psi (1) = T_3 f - f \text{ and } T_2 \psi (1) - \psi (1) = T_3 f' - f', \quad \text{where $1$ is here viewed in }\langle T_3\rangle. 
$$
\item In the notations of the first item, the element $\hat\partial [\psi]$ is the class in $[C(\Sigma, \Z)^{\langle T_3\rangle}]_{{\langle T_1, T_2\rangle}}$ of the $\langle T_3\rangle$-invariant  element 
$$
T_2 f - f -  T_1f' + f'.
$$
\end{enumerate}
\end{lemma}

{{
\begin{proof}
For the first item, we notice that by definition, any element of $\left( H^1 (\langle T_3\rangle, C(\Sigma, \Z)\right)^{\langle T_1, T_2\rangle}$ is a class $[\psi]$ of a $1$-cocycle $\psi\in Z^1 (\langle T_3\rangle, C(\Sigma, \Z)$ such that $T_1\psi - \psi$ and $T_2\psi - \psi$ are coboundaries for the $\langle T_1\rangle$-action. Now, such $1$-cocycle $\psi$ is totally determined by its value at $1\in  \langle T_3\rangle$ and it is then easy to check that the above conditions coincide exactly with the assumption of existence of $f$ and $f'$ satisfying 
$$
T_1 \psi (1) - \psi (1) = T_3 f - f \text{ and } T_2 \psi (1) - \psi (1) = T_3 f' - f', \quad \text{where $1$ is here viewed in }\langle T_3\rangle. 
$$
Notice first that $T_2 f - f -  T_1f' + f'$ is $\langle T_3\rangle$-invariant, for setting $g:= \psi (1)$, we have using the commutation of the actions
\begin{eqnarray*}
T_3\left( T_2 f - f -  T_1f' + f' \right) & = & (T_2 - I) (f+T_1 g - g) - (T_1-I) (f'+T_2 g -g)\\
& = & T_2f -f  - T_1 f' + f'.
\end{eqnarray*}
To describe the boundary map, we introduce the $1$-cocycle  $\varphi \in Z^1 (\langle T_1, T_3\rangle, C(\Sigma, \Z))$ which satisfies:
$$
\varphi (1, 0) = T_3 f\text{ and } \varphi (0, 1) = T_3 g=T_3 \psi (1).
$$
The explicit formula for $\varphi$ is then obvious and we have for instance when $n_1, n_3\geq 1$:
$$
\varphi (n_1, n_3) = T_1^{n_1}\sum_{k=1}^{n_3} T_3^k g + T_3 \sum_{k=0}^{n_1-1}T_1^k f,
$$
and a similar explicit formula for any $(n_1, n_3)\in \langle T_1, T_3\rangle$. It is then straightforward to show that the class of $\varphi$  is a preimage of $[\psi]$. Now, the map
$$
\langle T_2\rangle \ni n_2 \longmapsto [T_2^{n_2} \varphi  - \varphi]\in H^1(\langle T_1, T_3\rangle, C(\Sigma, \Z)),
$$
is clearly valued in the range of the monomorphism $i'$ of the exact sequence \eqref{ExactSequence2}. Hence, we get in this way a representative for a class in 
$$
H^1 ( \langle T_2\rangle, H^1 (\langle T_1\rangle, C(\Sigma, \Z)^{\langle T_3\rangle}) \simeq [C(\Sigma, \Z)^{\langle T_3\rangle}]_{{\langle T_1, T_2\rangle}}.
$$
The last isomorphism is evaluation at $1$ in $\langle T_2\rangle$ followed by evaluation at $1$ in $\langle T_1\rangle$. To conclude the proof we need a $\langle T_3\rangle$-invariant representative of
$$
  T_2 \varphi (1, 0) - \varphi (1, 0)
  =  T_2 T_3 f - T_3 f 
  $$
But,
 \begin{eqnarray*} 
T_2 T_3 f - T_3 f - [T_1 (T_3f') - T_3f'] & = &  (T_2 - I) (f+T_1g -g) - (T_1 -I) (f' +T_2 g -g) \\
& = & T_2f - T_1 f' -f +f'
\end{eqnarray*}
Hence the $\langle T_3\rangle$ -invariant element $T_2f - T_1 f' -f +f'$ is such a representative and  represents the class $\hat\partial [\psi]$ as announced.
\end{proof}
}}

\begin{theorem}\label{3dConj2}
{{Assume that the natural map 
$$
\left[ C(\Sigma, \Z)^{\langle T_3\rangle}\right]_{\langle T_1, T_2\rangle} \rightarrow C(\Sigma, \Z)_{\langle T_1, T_2\rangle}
$$
induced by the inclusion $C(\Sigma, \Z)^{\langle T_3\rangle}\hookrightarrow C(\Sigma, \Z)$, is injective, then the group $\Z_{12}[\mu]$ is contained in the magnetic gap-labelling group.}}
\end{theorem}

\begin{proof}\
{{This is an easy corollary of the previous lemma. It is clear from the previous lemma that the composition  of the boundary map $\hat\partial$ with the map 
$$
\left[ C(\Sigma, \Z)^{\langle T_3\rangle}\right]_{\langle T_1, T_2\rangle} \rightarrow C(\Sigma, \Z)_{\langle T_1, T_2\rangle}
$$ 
is the zero map. Therefore, under the assumption of Theorem \ref{3dConj2}, we deduce that the boundary map must be the zero map. Applying the exactness of the cohomology exact sequence, we deduce that the restriction of the epimorphism $\pi'$ in the exact sequence \eqref{ExactSequence2} to the $\langle T_2\rangle$-invariants is hence still an epimorphism onto the $\langle T_2\rangle$-invariants. Therefore, the subgroup $A$ of Lemma \ref{beta} coincides with the whole  group $[C(\Sigma, \Z)_{\langle T_3\rangle}]^{\langle T_1, T_2\rangle}$ and this finishes the proof.}}
\end{proof}

{{ We are now in position to deduce the proof of  Conjecture \ref{mainconj-minimal} in the 3D case.}}

{{\begin{corollary}\label{3Dconj2}
\begin{enumerate}
\item Assume that the group $\langle T_3\rangle$ acts minimally, then the group $\Theta_{12} \Z_{12}[\mu]$ is contained in the magnetic gap-labelling group. 
\item Assume that our action of $\Z^3$ on $\Sigma$ is strongly minimal, then the magnetic gap-labelling group coincides with the magnetic frequency group, i.e. with 
$$
\Z[\mu] \; + \; \Theta_{12} \Z_{12}[\mu] \; + \; \Theta_{13} \Z_{13}[\mu] \; + \; \Theta_{23} \Z_{23}[\mu].
$$
\end{enumerate}
\end{corollary}}}

\begin{proof}
{{We need to show  that when the group $\langle T_3\rangle$ acts minimally on $\Sigma$, the natural map 
$$
\left[ C(\Sigma, \Z)^{\langle T_3\rangle}\right]_{\langle T_1, T_2\rangle} \rightarrow C(\Sigma, \Z)_{\langle T_1, T_2\rangle}
$$
is a monomorphism and apply Theorem \ref{3dConj2}. But notice that if the action of $\langle T_3\rangle$ is minimal, all elements of $C(\Sigma, \Z)^{\langle T_3\rangle }$ are constant integer valued functions given by $n \times \chi$ where $n\in \Z$ and $\chi$ is the constant function with value $1$ on $\Sigma$. Assume then that the image of the class $[n \times \chi]\in \left[ C(\Sigma, \Z)^{\langle T_3\rangle}\right]_{\langle T_1, T_2\rangle}$ under the above map is zero in $C(\Sigma, \Z)_{\langle T_1, T_2\rangle}$, then its integral against $\mu$ must be trivial. Since $\mu$ is a probability measure, this implies in turn that $n=0$. }}

{{The second item is clear since we can permute the roles of the generators $T_1, T_2, T_3$. }}

\end{proof}


\section{{{An explicit construction for the ``easy-half''}}}\label{3D1}


The main result of this section is  Theorem \ref{Morphism} which allows in particular to deduce equality in Conjecture \ref{mainconj} under a technical hypothesis on the given tiling, see Proposition \ref{EasyInclusion} and Definition \ref{defn:hypothesisH}. As explained previously, the inclusion of the magnetic frequency group in the magnetic gap-labelling group is expected to hold under suitable dynamical conditions on the given aperiodic  tiling. In the present section, given a coinvariant class $f\in C(\Sigma, \Z)_{\Z^{I^c}}$ which is $\Z^I$-invariant, we prove under a suitable combinatorial assumption on representatives of $f$, that its integral against the probability measure $\mu$, multiplied by the Pfaffian of $\Theta_I$, belongs   to the magnetic gap-labelling group.  The proof relies on the existence, for any multiplier $\sigma$ of the subgroup $\Z^I$ of $\Z^p$, of a commutative diagram (see Theorem \ref{Morphism} again):
$$
\begin{CD}
K_0( C^* \Z^{I}  , \sigma)@>  {\Phi_{f, *}}  >> K_0( C(\Sigma) \rtimes_{i_*\sigma} \Z^p) \\
@V{\tau_*}VV     @VV {\tau^{\mu}_*} V\\
\R  @> {\mu(f) \times \bullet} >> \R
\end{CD}
$$
Such commutative diagram can be interpreted as a twisted (non-smooth) version of the classical Morita extension map associated with a given transversal in a foliation \cite{BenameurHeitsch, ConnesSkandalis}. 
For the clarity of the exposition, we have restricted ourselves to the 3D case, where the construction is already  technically involved.  The assumption on the class $f$ allows us to reduce the problem to classes represented by characteristic classes of clopen subspaces which live in the magnetic frequency group and the statement of Theorem \ref{Morphism} corresponds to such clopen subspaces.  The main step in the proof of Theorem \ref{Morphism}  is Proposition \ref{TechnicalProposition} which allows one to construct an explicit Morita morphism between the relevant $C^*$-algebras, and the corresponding commutative diagram then follows immediately. 

Fix a clopen subspace $\Lambda$ of the Cantor space $\Sigma$ such that the image of the characteristic function of $\Lambda$ 
in the $\langle T_3\rangle$-coinvariants, 
is a $\langle T_1, T_2\rangle$-invariant class.  In the following definition and in the major part of this section, we have given a specific role to the third generator $T_3$, but the similar constructions and proofs work if we operate any permutation of the generators $T_1$, $T_2$ and $T_3$.

\begin{definition}\label{defn:hypothesisH}
The  clopen subspace $\Lambda$  satisfies Hypothesis (H) if we can decompose $\Lambda$ into clopen subsets $(K_{i})_{1\leq i \leq q}$
$$
\Lambda = (K_{1} \amalg \cdots \amalg K_{r})\amalg (K_{{r+1}} \amalg \cdots \amalg K_{q}) = K \amalg (K_{{r+1}} \amalg \cdots \amalg K_{q}),
$$
such that
\begin{itemize}
\item For $r+1\leq i \leq q$, there exists $j_3(i)\in \{1, \cdots, r\}$ such that $K_{i}=T_3^{\beta_i} K_{{j_3(i)}}$ for some $\beta_i\in \Z$. In particular, $\langle T_3\rangle \Lambda = \langle T_3\rangle K$.
\item $
\langle T_3 \rangle  K_{i} \cap \langle T_3 \rangle  K_{j} = \emptyset \text{ for } 1\leq i\neq j \leq r.$
\item For any $1\leq i\leq q$, $\exists (j_1(i), j_2(i))\in \{1, \cdots , r\}^2$ and $(k_1(i), k_2(i))\in \Z^2$ such that  
$$
T_1  (K_{i}) = T_3^{k_1(i)} \left(K_{{j_1(i)}}\right) \text{ and } T_2  (K_{i}) = T_3^{k_2(i)}\left( K_{{j_2(i)}}\right), \quad 1\leq i \leq q.
$$
\end{itemize}
\end{definition}

So, this means more specifically that there are (unique) surjections $j_1, j_2: \{1, \cdots, q\} \to \{1, \cdots ,r\}$ whose restrictions to $\{1, \cdots, r\}$ are permutations.
Notice that  we get a well defined map $j_3:\{r+1, \cdots, q\}\rightarrow \{1, \cdots, r\}$ that we shall extend to $\{1, \cdots, q\}$ by setting $j_3(i):=i$ and $\beta_i=0$ if $1\leq i\leq r$. Moreover, the values of $j_1(i)$ and $j_2(i)$ for $i=r+1, \cdots, q$ are prescribed by the values on $\{1, \cdots, r\}$ since we must have 
$$
j_1 = j_1 \circ j_3 \text{ and } j_2=j_2\circ j_3\text{ on } \{r+1, \cdots, q\}. 
$$
Since the projection of the characteristic function of $\Lambda$ in the coinvariants modulo $\langle T_3 \rangle $ is $\langle T_1, T_2 \rangle $-invariant, the cardinal $\varphi_j$ of $j_3^{-1} (j)$ is automatically constant on each orbit under $\langle j_1, j_2 \rangle $.

We shall concentrate on the restricted permutations to $\{1, \cdots, r\}$ and then extend the constructions to $\{1, \cdots, q\}$. It is important in the sequel that we can exploit the relative freeness in the choice of the  integer valued maps $k_1$ and $k_2$.  An easy consequence of the definitions is that the two permutations $j_1$ and $j_2$ of $\{1, \cdots, r\}$ commute. More precisely, notice that since $T_1T_2=T_2T_1$, we have by definition of $j_1$ and $j_2$ that
$$
\langle T_3 \rangle  K_{{j_1j_2(i)}} = \langle T_3 \rangle  K_{{j_2j_1(i)}}.
$$
But for $1\leq i\leq r$ we know that the orbits under $\langle T_3 \rangle $ of the clopen sets $K_{j}$ are disjoint. Hence necessarily $j_2j_1(i)=j_1j_2(i)$. In fact, we easily see that this commutation relation holds on $\{1, \cdots, q\}$.

The goal of this section is to prove the following theorem and to explain its relation with the easy-half of the conjecture.

\begin{theorem}\label{Morphism}
For any clopen $\Lambda$ in $\Sigma$ as above which further satisfies Hypothesis (H) and any multiplier $\sigma$ of the group $\langle T_1, T_2\rangle$, there exists a $C^*$-algebra homomorphism
$$
\Phi_\Lambda : C^* (\langle T_1, T_2 \rangle , \sigma) \longrightarrow M_\infty( C(\Sigma) \rtimes_{i_*\sigma} \Z^3),
$$
such that the following diagram commutes
$$
\begin{CD}
K_0( C^* (\langle T_1, T_2 \rangle , \sigma)@>  {\Phi_{\Lambda, *}}  >> K_0( C(\Sigma) \rtimes_{i_*\sigma} \Z^3) \\
@V{\tau_*}VV     @VV {\tau^{\mu}_*} V\\
\R  @> {\mu(\Lambda) \times \bullet} >> \R
\end{CD}
$$
The same statement holds after any permutation of the generators $T_1$, $T_2$ and $T_3$.
\end{theorem}

The proof of this theorem will occupy the rest of this section and will be split into many lemmas. Since the generators play a perfect symmetric role, once the proof is given with the special role of $T_3$, it will hold immediately for all permutations of the generators.

\begin{lemma}
With the previous notations, there exist integer  valued maps $k_1, k_2: \{1, \cdots, q\} \rightarrow \Z$ such that for any $i=1, \cdots, q$:
\begin{itemize}
\item $T_1 T_3^{-k_1(i)} (K_{i}) = K_{{j_1(i)}}$  and $T_2 T_3^{-k_2(i)} (K_{i}) = K_{{j_2(i)}}$.
\item The following relations hold:
\begin{equation}\label{compatibility}
k_2 (j_1(i)) - k_2 (i) = k_1 (j_2(i)) - k_1 (i).
\end{equation}
\end{itemize}
\end{lemma}

\begin{proof}\
We first concentrate on $\{1, \cdots, r\}$ and will extend $k_1$ and $k_2$ later on. The relation \eqref{compatibility} must be satisfied independently on every orbit of $j_1$ and $j_2$ of the form
$$
A := \{j_1^{l_1} j_2^{l_2} (\lambda), l_1, l_2\in \Z\}, \text{ for a given } \lambda\in \{1, \cdots, r\}.
$$
So we only need to give the construction for one such orbit. We denote by $p_1\geq 1$ and $p_2\geq 1$ the respective orders of $\lambda$ with respect to $j_1$ and $j_2$ and we set
$$
A_{l_2}:=\{j_2^{l_2} (\lambda), j_1 j_2^{l_2} (\lambda), \cdots, j_1^{p_1-1}j_2^{l_2} (\lambda)\}\text{ so that } A= \cup_{0\leq l_2\leq p_2-1} A_{l_2}
$$
We point out that since $j_1$ and $j_2$ commute, the order of all elements of $A_0$ under  $j_2$ is equal to $p_2$. Let $\varrho_2$ be the global order of $A_0$, that is the least integer $\varrho\geq 1$ such that $\varrho (A_0)=A_0$. Then $\varrho_2\leq p_2$ and for $0\leq l_1\leq p_1-1$ and $1\leq l_2\leq \varrho_2-1$, the integers  $j_1^{l_1} j_2^{l_2} (\lambda)$ are all distinct from each other so that the first item of the lemma would be easy to satisfy in the sequel, and there is a unique  integer $0\leq \varrho_1 \leq p_1-1$ such that  $j_2^{\varrho_2} (\lambda) = j_1^{\varrho_1} (\lambda)$. Notice also that if we write each $A_{l_2}$ in the following order
$$
A_{l_2} = \{j_2^{l_2}(\lambda), \cdots, j_2^{p_2-1}(\lambda), \lambda, \cdots, j_2^{l_2-1}(\lambda)\},
$$
then $j_2^{\varrho_2}$ is nothing but the $\varrho_1$ power of the cyclic permutation of $p_1$ variables.  We now construct $k_1(l_1, l_2):= k_1( j_1^{l_1} j_2^{l_2} (\lambda))$ on $[0, p_1-1]\times [0, \varrho_2-1]$. Assume that $k_1$ is given arbitrarily on $A_0$ and on any $A_{l_2}\smallsetminus \{j_1^{p_1-1}j_2^{l_2}(\lambda)\}$ satisfying the first item of the  lemma. More precisely, we assume that $k_1(l_1, l_2)$ are given integers for $l_2=0$ and $0\leq l_1\leq p_1-1$ on the one hand and for $1\leq l_2 \leq \varrho_2-1$ and $0\leq l_1\leq p_1-2$ on the other hand, so that they satisfy 
$$
T_1 T_3^{-k_1(l_1, l_2)} \left(K_{{j_1^{l_1} j_2^{l_2} (\lambda)}}\right) = K_{{j_1^{l_1+1} j_2^{l_2} (\lambda)}}\text{ and }T_2 T_3^{-k_2(l_1, l_2)} \left(K_{{j_1^{l_1} j_2^{l_2} (\lambda)}}\right) = K_{{j_1^{l_1} j_2^{l_2+1} (\lambda)}}.
$$
The second item actually imposes the missing values $k_1 (p_1-1, l_2)$ of $k_1$. More precisely, the sum $C_{l_2}:=\sum_{0\leq l_1\leq p_1-1} k_1( l_1, l_2)$ is then necessarily constant in $l_2$ and thus equal to $C_0$, for by \eqref{compatibility}
\begin{eqnarray*}
\sum_{0\leq l_1\leq p_1-1} [k_1( l_1, l_2+1) - k_1( l_1, l_2)] & = & \sum_{0\leq l_1\leq p_1-1} [k_2 ( l_1+1, l_2) - k_2( l_1, l_2)] \\
& = & k_2 (p_1, l_2) - k_2 (0, l_2) = 0.
\end{eqnarray*}
We thus set for $1\leq l_2\leq \varrho_2-1$ 
$$
k_1 (p_1-1, l_2) := C_0 - \sum_{0\leq l_1\leq p_1-2} [k_2 ( l_1+1, l_2).
$$
An easy verification shows that $k_1 (p_1-1, l_2)$ then satisfies the first item. Indeed, notice that
$$
T_1^{-1} T_3^{k_1(p_1-2, l_2)} (K_{{j_1^{p_1-1}j_2^{l_2}(\lambda)}}) = K_{{j_1^{p_1-2}j_2^{l_2}(\lambda)}}
$$
and similarly for $(p-3, l_2)$ etc. Therefore,
\begin{eqnarray*}
T_1T_3^{-k_1(p_1-1, l_2)} (K_{{j_1^{p_1-2}j_2^{l_2}(\lambda)}}) & = & T_1^{-(p_1-1)} T_3^{k_1(0, l_2) + \cdots + k_1(p_1-2, l_2)} \; ( K_{{j_1^{p_1-2}j_2^{l_2}(\lambda)}})\\
& = & \left[T_1^{-1} T_3^{k_1(0, l_2)}\right]\cdots  \left[T_1^{-1} T_3^{k_1(p_2-2, l_2)}\right] \left(K_{{j_1^{p_1-1}j_2^{l_2}(\lambda)}} \right)\\
& = & K_{{j_2^{l_2}(\lambda)}}
\end{eqnarray*}
It is easy to check that no other  condition is imposed on the values of $k_1$ on $[0, p_1-1]\times [0, \varrho_2-1]$ by the compatibility condition \eqref{compatibility}. Hence  
$k_1$ is now well defined on $[0, p_1-1]\times [0, \varrho_2-1]$ and satisfies the first item of the lemma. It is then extended to $[0, p_1-1]\times [0, p_2-1]$ by using that $\varrho_2$ is a divisor of $p_2$ and that $k_1$ (and also $k_2$) must satisfy that for any integer $\nu$
$$
k_1 (l_1, l_2+\nu \varrho_2) := k_1 (l_1+\nu \varrho_1, l_2).
$$
That such extension of $k_1(l_1, l_2+\nu \varrho_2)$ still satisfies the first item of the lemma is again straightforward since $k_1 (l_1+\nu \varrho_1, l_2)$ does for $l_2\leq \varrho_2-1$ and since 
$$
j_1^{l_1+\nu \varrho_1} j_2^{l_2} (\lambda) = j_1^{l_1} j_2^{l_2+\nu \varrho_2} (\lambda). 
$$
We shall now impose the compatibility condition \eqref{compatibility} to deduce $k_2$. Again, we need first to choose the values $k_2(0, l_2)$ for any $l_2\in [0, \varrho_2-1]$ which turn out to be arbitrary as far as they satisfy the first item, and we now show that all the values of $k_2(l_1, l_2)$ are  prescribed on $[0, p_1-1]\times [0, p_2-1]$ and satisfy the lemma. This is done on $[0, p_1-1]\times [0, \varrho_2-1]$ and then deduced again by the relation 
$$
k_2 (l_1, l_2+\nu \varrho_2) := k_2 (l_1+\nu \varrho_1, l_2).
$$
We proceed, for $0\leq l_2\leq \varrho_2-1$, inductively on  $l_1$. We set for instance, 
$$
k_2(1, l_2) := k_2(0, l_2) +k_1(0, l_2+1) - k_1 (0, l_2).
$$
Again such expression automatically satisfies the first item of the lemma. We then repeat the process for the induction in $l_1$ and deduce the values of $k_2(l_1, l_2)$ by using
$$
k_2(l_1+1, l_2) - k_2(l_1, l_2)  = k_1(l_1, l_2+1) - k_1 (l_1, l_2).
$$
We notice that as for $k_1$, the values of the integers $k_2(l_1, \varrho_2-1)$ could as well be deduced from the relation:
$$
\sum_{0\leq l_2\leq \varrho_2-1} k_2(l_1+1, l_2) - \sum_{0\leq l_2\leq \varrho_2-1} k_2(l_1, l_2) =  k_1(l_1+\varrho_1, 0) - k_1(l_1, 0).
$$
This is however compatible with the previous definition and is redundant. 

We conclude by explaining how to extend the maps $k_1$ and $k_2$ defined so far on $\{1, \cdots ,r\}$ only, to $\{1, \cdots, q\}$. We set for $r+1\leq i\leq q$:
$$
k_1 (i) := k_1 (j_3(i)) + \beta_i \text{ and } k_2 (i) := k_2 (j_3(i)) + \beta_i
$$
Then an easy verification shows that the extended maps  also satisfy the relations. More specifically, we can write for $r+1\leq i \leq q$:
$$
k_2 (j_1(i)) - k_2 (i) = [k_2 (j_1(j_3(i))) - k_2 (j_3(i))]-\beta_i = [k_1(j_2(j_3(i))) - k_1 (j_3(i))] - \beta_i = k_1(j_2(i)) - k_1(i).
$$
Moreover,
$$
T_1 K_{i} = T_1T_3^{\beta_i} (K_{{j_3(i)}}) = T_3^{\beta_i} T_3^{k_1 (j_3(i))} K_{j_1j_3(i)}=T_3^{k_1 (j_3(i))+\beta_i} K_{j_1(i)} = T_3^{k_1 (i)} K_{j_1(i)},
$$
and similarly for $T_2$.
\end{proof}

We denote for $r+1\leq i\leq q$ by $j_1^{-1} (i)$ the integer $j_1^{-1}(j_3(i))$ and similarly for $j_2$. So notice that 
$$
j_1^{-1}\circ j_1 = j_1\circ j_1^{-1} = j_3 \text{ and } j_2^{-1}\circ j_2 = j_2\circ j_2^{-1} = j_3,
$$ 
where $j_3$ has been extended to $\{1, \cdots, r\}$ by the identity map.

\begin{proposition}\label{TechnicalProposition}
Given $k_1$ and $k_2$ as in the previous lemma (so satisfying \eqref{compatibility}), the inductive relations 
\begin{equation}\label{induction}
k_2^{n_2+1}k_1^{n_1} (i) = k_2^{n_2}k_1^{n_1} (j_2(i)) + k_2(i)\;\; \text{ and } \;\; k_2^{n_2}k_1^{n_1+1} (i) = k_2^{n_2}k_1^{n_1} (j_1(i)) + k_1(i).
\end{equation}
together with $k_2^1k_1^0 = k_2$ and $k_2^0k_1^1 = k_1$ allow to {\underline{well define}} for any $(n_1, n_2)\in \Z^2$,  integer valued functions $k_2^{n_2}k_1^{n_1} : \{1, \cdots, q\} \rightarrow \Z$ satisfying
\begin{enumerate}
\item $\left( T_1^{n_1} T_2^{n_2} T_3^{-k_2^{n_2}k_1^{n_1} (i)} \right) (K_{i}) = K_{{j_2^{n_2}j_1^{n_1} (i)}}.$
\item $k_2^{m_2}k_1^{m_1} (j_2^{n_2-m_2} j_1^{n_1-m_1}(i))  + k_2^{n_2-m_2}k_1^{n_1-m_1} (i)  = k_2^{n_2}k_1^{n_1} (i) $, for any integers $n_1, n_2, m_1, m_2$ and any $i\in \{1, \cdots , q\}$.
\end{enumerate}
\end{proposition}

Notice that necessarily $k_2^0k_1^0=0$ and that the last item is trivially satisfied when $(m_1, m_2)=(n_1, n_2)$ or $(m_1, m_2)= (0, 0)$. Once such sequence $k_2^{n_2}k_1^{n_1}$ is proved to be well defined, one can deduce many expressions for it. For instance, for $n_1, n_2 > 0$, we have the following two expressions which then must fit by the previous proposition:
$$
k_2^{n_2}k_1^{n_1} (i) = \left[k_2(j_2^{n_2-1} j_1^{n_1} (i)) + \cdots + k_2(j_1^{n_1} (i))\right] + \left[ k_1(j_1^{n_1-1} (i)) + \cdots + k_1(i)\right]
$$
and 
$$
k_2^{n_2}k_1^{n_1} (i) = \left[k_1(j_1^{n_1-1} j_2^{n_2} (i)) + \cdots + k_2(j_2^{n_2} (i))\right] + \left[ k_2(j_2^{n_2-1} (i)) + \cdots + k_2(i)\right]
$$
It is also a straightforward exercise to check directly the coincidence of these two expressions.  There are similar expressions for negative $n_i$'s, and for any $n_i's$. When $n_1<0$ and $n_2<0$, one gets for instance an expression
$$
- k_2^{n_2}k_1^{n_1} (i)= \left[k_1(j_1^{n_1}j_2^{n_2}(i)) + \cdots + k_1(j_1^{-1}j_2^{n_2}(i))\right] + \left[ k_2 (j_2^{n_2}(i) + \cdots + k_2(j_2^{-1} (i)) \right]
$$
We shall first prove the well-definiteness of these maps and then use any of the above two expressions to deduce for instance the second item. 
\begin{proof}
We again concentrate on $\{1, \cdots, r\}$ where $j_1$ and $j_2$ are permutations and   explain later the different expressions  on $[r+1, q]$. The inductive definition allows to define many candidates for  $k_2^{n_2} k_1^{n_1}$ obtained by using different polygonal paths from $(0, 0)$ to $(n_1, n_2)$ in $\Z^2$, i.e. paths composed of horizontal and vertical segments with endpoints in $\Z^2$ and which start at $(0, 0)$ and end at $(n_1, n_2)$. Now, consider for any $(m_1, m_2)\in \Z^2$, the two paths joining $(m_1, m_2)$ to $(m_1+1, m_2+1)$ given by
\begin{multline*}
\maC_1:=[(m_1,m_2), (m_1+1, m_2)] \cup [(m_1+1, m_2), (m_1+1, m_2+1)] \quad \text{ and }\\ \quad \maC_2:=[(m_1,m_2), (m_1, m_2+1)] \cup [(m_1, m_2+1), (m_1+1, m_2+1)].
\end{multline*}
Applying the inductive argument with $\maC_1$ we get for any $i\in [1, q]$:
$$
k_2^{m_2+1}k_1^{m_1+1} (i) = k_2^{m_2}k_1^{m_1} (j_1j_2(i)) + k_1(j_2(i)) + k_2 (i).
$$
Now using $\maC_2$, we get
$$
k_2^{m_2+1}k_1^{m_1+1} (i) = k_2^{m_2}k_1^{m_1} (j_2j_1(i)) + k_2(j_1(i)) + k_1 (i).
$$
Therefore, and since we have chosen $k_1$ and $k_2$ so that $k_1(j_2(i)) + k_2 (i)=k_2(j_1(i)) + k_1 (i)$ and also since $j_1j_2=j_2j_1$, we see that we obtain the same result by using $\maC_1$ or $\maC_2$. We could as well use the inverse paths $-\maC_1$ and $-\maC_2$ and compute $k_2^{m_2}k_1^{m_1} (i)$ in terms of $k_2^{m_2+1}k_1^{m_1+1}$ and see that we also get the same result. Now if $\maC$ is any  polygonal path from the base point $(0, 0)$,  where $k_2^0k_1^0=0$, to $(n_1, n_2)$, composed of horizontal and vertical segments, we can use the inductive formulae to deduce an expression of $k_2^{n_2}k_1^{n_1} (i)$ which by repeating the previous argument as many times as necessary, will not depend on the chosen path. 

So, if we choose the simple path with one horizontal segment and one vertical segment $[(0, 0), (n_1, 0)]\cup [(n_1, 0), (n_1, n_2)]$ then we get the expression
$$
k_2^{n_2}k_1^{n_1} (i) = \left[k_1(j_1^{n_1-1} j_2^{n_2} (i)) + \cdots + k_2(j_2^{n_2} (i))\right] + \left[ k_2(j_2^{n_2-1} (i)) + \cdots + k_2(i)\right].
$$
while the other simple path $[(0, 0), (0, n_2)]\cup [(0, n_2), (n_1, n_2)]$ yields the expression
$$
k_2^{n_2}k_1^{n_1} (i) = \left[k_2(j_2^{n_2-1} j_1^{n_1} (i)) + \cdots + k_2(j_1^{n_1} (i))\right] + \left[ k_1(j_1^{n_1-1} (i)) + \cdots + k_1(i)\right].
$$

Let us now prove the first item of the proposition. If $n_1=n_2=0$ then the item is satisfied by obvious observation. We assume now that $k_2^{n_2}k_1^{n_1}(i)$ satisfies the first item for any $i=1, \cdots, r$. Then 
\begin{eqnarray*}
T_1^{n_1}T_2^{n_2+1} T_3^{-k_2^{n_2+1}k_1^{n_1} (i)} \left( K_{i} \right) & = & T_1^{n_1}T_2^{n_2} T_3^{-k_2^{n_2}k_1^{n_1} (j_2(i))} \left( T_2 T_3^{-k_2(i)} (K_{i}) \right)\\
& = & T_1^{n_1}T_2^{n_2} T_3^{-k_2^{n_2}k_1^{n_1} (j_2(i))}\left( K_{{j_2(i)}} \right)\\
& = & K_{{j_1^{n_1}j_2^{n_2+1} (i)}}
\end{eqnarray*}
In a similar way we prove that
$$
T_1^{n_1+1}T_2^{n_2} T_3^{-k_2^{n_2}k_1^{n_1+1} (i)} \left( K_{i} \right)  =  K_{{j_1^{n_1+1}j_2^{n_2} (i)}}.
$$
We obtain as well
\begin{eqnarray*}
T_1^{n_1}T_2^{n_2-1} T_3^{-k_2^{n_2-1}k_1^{n_1} (i)} \left( K_{i} \right) & = &  T_1^{n_1}T_2^{n_2} T_3^{-k_2^{n_2}k_1^{n_1} (j_2^{-1}(i))} \left( T_2^{-1} T_3^{k_2(j_2^{-1} (i))} (K_{i}) \right)\\
& = & T_1^{n_1}T_2^{n_2} T_3^{-k_2^{n_2}k_1^{n_1} (j_2^{-1}(i))} \left(  K_{{j_2^{-1}(i)}} \right)\\
& = & K_{{j_1^{n_1} j_2^{n_2-1}(i)}}
\end{eqnarray*}
and again similarly we get
$$
T_1^{n_1-1}T_2^{n_2} T_3^{-k_2^{n_2-1}k_1^{n_1} (i)} \left( K_{i} \right) = K_{{j_1^{n_1-1} j_2^{n_2}(i)}}.
$$
We hence get the first item for any $(n_1, n_2)\in \Z^2$.

We now prove the second item. We  prove first the following relation (which corresponds to the second item for $(n_1, n_2)=(0,0)$ and where we changed the notation):
$$
k_2^{n_2}k_1^{n_1} (j_2^{-n_2}j_1^{-n_1} (i)) + k_2^{-n_2}k_1^{-n_1} (i) = 0, \quad (n_1,n_2)\in \Z^2.
$$
For $n_1, n_2\geq 0$ we just use the expression
\begin{eqnarray*}
k_2^{n_2}k_1^{n_1} (j_2^{-n_2}j_1^{-n_1} (i)) &=& k_2(j_2^{n_2-1}j_1^{n_1} (j_2^{-n_2}j_1^{-n_1} (i)) ) + \cdots + k_2 (j_1^{n_1}j_2^{-n_2}j_1^{-n_1}(i)) \\ & & + k_1(j_1^{n_1-1}j_2^{-n_2}j_1^{-n_1}(i)) + \cdots + k_1(j_2^{-n_2}j_1^{-n_1} (i))\\
& = & k_2(j_2^{-1}(i)+ \cdots + k_2(j_2^{-n_2} (i))+ k_1(j_1^{-1}j_2^{-n_2}(i))+ \cdots + k_1(j_1^{-n_1}j_2^{-n_2}(i))
\end{eqnarray*}
But this is precisely the expression that we already got for $k_2^{-n_2}k_1^{-n_1} (i)$ in this case.   If $n_1<0$ and $n_2\geq 0$ then a direct computation gives
\begin{multline*}
k_2^{n_2}k_1^{n_1} (j_2^{-n_2}j_1^{-n_1} (i)) = \left[ k_2(j_2^{-n_2} j_1^{-n_1} (i)) + \cdots + k_2(j_2^{-1} j_1^{-n_1} (i))\right]\\ - \left[ k_1(i) + \cdots + k_1( j_1^{-n_1-1}(i)) \right]
\end{multline*}
Computing $k_2^{-n_2}k_1^{-n_1} (i)$ by applying first the induction to the positive integer $-n_1$, we get exactly the opposite to this expression. 
Notice now that by applying the previous results to $(-n_1, -n_2)$ and to $i'=j_2^{n_2}j_1^{n_1} (i)$ we see that the formula for $(n_1, n_2)$ is equivalent to the formula for $(-n_1, -n_2)$. Hence we have proved the formula of the second item when $n_1= n_2= 0$. Assume that this formula is satisfied for a given $(n_1, n_2)\in \Z^2$ and for any $i\in \{1, \cdots, r\}$ and any $(m_1, m_2)\in \Z^2$, then 
\begin{eqnarray*}
k_2^{n_2+1}k_1^{n_1} (i) & = & k_2^{n_2}k_1^{n_1} (j_2(i)) + k_2(i)\\
 &= & k_2^{m_2}k_1^{m_1} (j_2^{n_2-m_2}j_1^{n_1-m_1} (j_2(i))) + \left( k_2^{n_2-m_2}k_1^{n_1-m_1} (j_2(i)) + k_2(i)\right) \\
 &= & k_2^{m_2}k_1^{m_1} (j_2^{n_2+1-m_2}j_1^{n_1-m_1} (i)) + k_2^{n_2-m_2+1}k_1^{n_1-m_1} (i).
\end{eqnarray*}
Hence the formula is satisfied for $(n_1, n_2+1)$. We leave it to the interested reader to check that the formula is then also satisfied for $(n_1+1, n_2)$, $(n_1-1, n_2)$  as well as for $(n_1, n_2-1)$ by applying each time the induction formula.

So far we have concentrated on $i\in \{1, \cdots, r\}$ and we now  give the expressions of the maps $k_2^{n_2}k_1^{n_1}$ on $\{r+1, \cdots, q\}$ which extend the proposition in a straightforward manner. This is easy since we just set for $r+1\leq i\leq q$:
$$
k_2^{n_2}k_1^{n_1} (i) := k_2^{n_2}k_1^{n_1} (j_3(i)) + \beta_i.
$$
This expression satisfies again the inductive relations since
$$
k_2^{n_2+1}k_1^{n_1} (i) = k_2^{n_2+1}k_1^{n_1} (j_3(i)) + \beta_i = k_2^{n_2}k_1^{n_1} (j_2j_3(i)) + k_2(j_3(i)) + \beta_i=k_2^{n_2}k_1^{n_1} (j_2(i)) + k_2(i),
$$
and similarly for $k_2^{n_2}k_1^{n_1+1} (i)$. Also, we have $k_2^1k_1^0 (i)= k_2(j_3(i))+\beta_i= k_2(i)$ and $k_2^0k_1^1 (i)= k_1(j_3(i))+\beta_i= k_1(i)$. Moreover, if $(m_1, m_2)\neq (n_1, n_2)$ and $r+1\leq i\leq q$ then
\begin{eqnarray*}
k_2^{m_2}k_1^{m_1} (j_2^{n_2-m_2}j_1^{n_1-m_1} (i)) + k_2^{n_2-m_2}k_1^{n_1-m_1} (i) & = & k_2^{m_2}k_1^{m_1} (j_2^{n_2-m_2}j_1^{n_1-m_1} (j_3(i))) + k_2^{m_2}k_1^{m_1} (j_3(i)) + \beta_i\\
& = & k_2^{n_2}k_1^{n_1} (j_3(i)) + \beta_i\\
& = & k_2^{n_2}k_1^{n_1} (i)
\end{eqnarray*}
Notice that we have used as before the convention $j_2^0=j_1^0=j_3$. 
\end{proof}

We are now in a position to prove our main theorem.

\begin{proof} (of Theorem \ref{Morphism})

We know that $\Lambda=\amalg_{1\leq i \leq q} K_{i}$ as before, so where the first $r$ indices represent all the disjoint orbits of $\Lambda$ under $T_3$ with the maps $j_1, j_2, j_3$ as well as the family of maps $k_2^{n_2}k_1^{n_1}$ constructed in the previous proposition. We denote for any $j\in \{1, \cdots, q\}$ by $\varphi_j\in \{1, \cdots, q-r+1\}$ the cardinal of the set $j_3^{-1}\{j_3(j)\}$ and we set $\varphi_0=0$ and $\hat\varphi_j:=\sum_{0\leq i\leq j} \varphi_i$.   We then define the diagonal projection matrix $\hat\chi_\Lambda$ in $M_{\hat\varphi_q} (C(\Sigma)\rtimes_{i_*\sigma} \Z^3)$ by setting for $\hat\varphi_{j-1}+1\leq h \leq \hat\varphi_j$:
$$
(\hat\chi^\Lambda)_{hh'} (n_1, n_2, n_3) := \delta_{h,h'} \delta_{(n_1, n_2, n_3), (0, 0, 0)}\cdot \chi_{K_{{j}}}.
$$
where $\chi_{K_{{j}}}$ is the characteristic function of the minimal clopen $ K_{{j}}$ and $\delta$ stands as usual for the Kronecker symbol. If $F\in \C[\langle T_1, T_2 \rangle , \sigma]$ is a finitely supported function and if $1\leq j, j'\leq q$, then we set
$$
{\widetilde F}^\Lambda_{jj'} (n_1, n_2, n_3) := \delta_{j_3(j), j_2^{n_2}j_1^{n_1} (j')} \delta_{n_3, -k_2^{n_2}k_1^{n_1}(j')+\beta_j} \times F(n_1, n_2)
$$
where $\beta_j$ was defined so that $T_3^{\beta_j} (K_{{j_3(j)}}) = K_{j}$ with the convention that for $1\leq j\leq r$, $\beta_j=0$. 
Now, we can define the matrix ${\hat F}^\Lambda \in M_{\hat\varphi_q} (\C[\Z^3, i_*\sigma])$ by setting
$$
{\hat F}^\Lambda_{hh'} := \frac{1}{\sqrt{\varphi_{j'}\varphi_j}} \delta_{h-\hat\varphi_{j-1}, h'-\hat\varphi_{j'-1}}  {\widetilde F}^\Lambda_{jj'} \text{ if }\hat\varphi_{j-1}+1\leq h\leq \hat\varphi_j\text{ and } \hat\varphi_{j'-1}+1\leq h'\leq \hat\varphi_{j'}.
$$
Notice that this is possibly non zero only when $j_3(j) = j_2^{n_2}j_1^{n_1} (j')$ and in this case $\varphi_j=\varphi_{j'}$. We now set
$$
\Phi_\Lambda  (F) := {\hat F}^\Lambda \star \hat\chi^\Lambda \quad \in M_{\hat\varphi_q} (C(\Sigma)\rtimes_{i_*\sigma} \Z^3),
$$
where  $\star$ is the product in $M_{\hat\varphi_q} ( C(\Sigma) \rtimes_{i_*\sigma} \Z^3)$.
Notice that each ${\hat F}^\Lambda_{hh'}$ is clearly finitely supported in $\Z^3$. In order to show that $\Phi_\Lambda$ extends to  a $*$-homomorphism and using that $ \hat\chi^\Lambda$ is a self-adjoint idempotent, we will  check the following relations on $\C[\langle T_1, T_2 \rangle , \sigma]$:
$$
{\hat F}^\Lambda \star \hat\chi_\Lambda =  \hat\chi_\Lambda\star {\hat F}^\Lambda, \; ({\hat F}^\Lambda)^*={\widehat {F^*}}^\Lambda\; \text{ and } \; {\hat F}^\Lambda \star {\hat G}^\Lambda = {\widehat{FG}}^\Lambda.
$$
Computing for $\hat\varphi_{j-1}+1\leq h\leq \hat\varphi_j$ and $\hat\varphi_{j'-1}+1\leq h'\leq \hat\varphi_{j'}$, we get using that $\sigma ((0,0), (n_1, n_2))=1$:
\begin{eqnarray*}
({\hat F}^\Lambda \star \hat\chi^\Lambda)_{hh'} (n_1, n_2, n_3) & = & \frac{1}{\sqrt{\varphi_{j'}\varphi_j}} \delta_{h-\hat\varphi_{j-1}, h'-\hat\varphi_{j'-1}} \times {\widetilde F}^\Lambda_{jj'} (n_1, n_2, n_3) T_1^{n_1} T_2^{n_2} T_3^{n_3} (\chi_{K_{{j'}}})\\
& = & \frac{1}{\sqrt{\varphi_{j'}\varphi_j}} \delta_{h-\hat\varphi_{j-1}, h'-\hat\varphi_{j'-1}} F(n_1, n_2) \delta_{j_3(j), j_2^{n_2}j_1^{n_1}(j')} \delta_{n_3, -k_2^{n_2}k_1^{n_1}(j')+\beta_j}  \chi_{K_{j}}
\end{eqnarray*}
The last equality is a consequence of the relations
\begin{multline*}
\delta_{j_3(j), j_2^{n_2}j_1^{n_1}(j')} T_1^{n_1}T_2^{n_2}T_3^{-k_2^{n_2}k_1^{n_1} (j')+\beta_j} (\chi_{K_{{j'}}}) = \delta_{j_3(j), j_2^{n_2}j_1^{n_1}(j')} T_3^{\beta_j} \chi_{K_{{j_2^{n_2}j_1^{n_1} (j')}}}\\ = \delta_{j_3(j), j_2^{n_2}j_1^{n_1}(j')} T_3^{\beta_j} (\chi_{K_{{j_3(j)}}})=\delta_{j_3(j), j_2^{n_2}j_1^{n_1}(j')} (\chi_{K_{j}})
\end{multline*}
Computing $(\hat\chi^\Lambda \star {\hat F}^\Lambda)_{hh'} (n_1, n_2, n_3)$, we get the same expression, so that 
$$
\hat\chi^\Lambda \star {\hat F}^\Lambda = {\hat F}^\Lambda \star \hat\chi^\Lambda.
$$
On the other hand, for a given extra element $G\in \C[\langle T_1, T_2 \rangle , \sigma]$,  we compute similarly using the same notations
\begin{eqnarray*}
({\hat F}^\Lambda \star {\hat G}^\Lambda )_{hh'} (n_1, n_2, n_3) & = & \frac{1}{\sqrt{\varphi_{j'}\varphi_j}} \sum_{j''=1}^q \sum_{h''=\hat\varphi_{j''-1}+1}^{\hat\varphi_{j''}} \delta_{h-\hat\varphi_{j-1}, h''-\hat\varphi_{j''-1}}  \delta_{h'-\hat\varphi_{j'-1}, h''-\hat\varphi_{j''-1}}  \\ & & \sum_{m_1, m_2} \frac{1}{\varphi_{j''}}\delta_{j_3(j), j_2^{m_2}j_1^{m_1} (j'')} \delta_{j_3(j''), j_2^{n_2-m_2}j_1^{n_1-m_1} (j')} \\
& & \delta_{m_3, -k_2^{m_2}k_1^{m_1}(j'') + \beta_j} \delta_{n_3-m_3, -k_2^{n_2-m_2}k_1^{n_1-m_1}(j') + \beta_{j''}}  \\ & & F(m_1, m_2) G(n_1-m_1, n_2-m_2) \sigma ((m_1, m_2), (n_1-m_1, n_2-m_2))\\
& = &  \frac{1}{\sqrt{\varphi_{j'}\varphi_j}} \delta_{h-\hat\varphi_{j-1}, h'-\hat\varphi_{j'-1}} \delta_{j, j_2^{n_2}j_1^{n_1} (j')} \delta_{n_3, -k_2^{n_2}k_1^{n_1} (j')+\beta_j} \\& &  \sum_{m_1, m_2} F(m_1, m_2) G(n_1-m_1, n_2-m_2)\\  & & \sigma((m_1, m_2), (n_1-m_1, n_2-m_2))   \sum_{j''=1}^q \frac{1}{\varphi_{j''}} \delta_{j_3(j''), j_2^{n_2-m_2}j_1^{n_1-m_1} (j')}.
\end{eqnarray*}
Indeed, recall that by the $\langle T_1, T_2 \rangle $-invariance of the class of the characteristic function of $\Lambda$ in the coinvariants with respect to $\langle T_3 \rangle $, we know that for any $n_1, n_2, m_1$ and $m_2$, and for any $j''\in j_3^{-1}(j_2^{n_2-m_2}j_1^{n_1-m_1} (j'))$, we have $
\varphi_{j''} = \varphi_{j_3(j'')} = \varphi _{j'} (=\varphi_j).$
But for any fixed $(m_1, m_2, n_1, n_2)\in \Z^4$ we have
$$
 \sum_{j''=1}^q \frac{1}{\varphi_{j''}} \delta_{j_3(j''), j_2^{n_2-m_2}j_1^{n_1-m_1} (j')} = \sum_{j_3(j'')=j_2^{n_2-m_2} j_1^{n_1-m_1} (j')} (\varphi_{j_2^{n_2-m_2} j_1^{n_1-m_1} (j')})^{-1} = 1.
$$
Therefore, we get for any $h, h'$:
$$
({\hat F}^\Lambda \star {\hat G}^\Lambda )_{hh'} = {\widehat{FG}}^\Lambda_{hh'}.
$$
using the previous results we deduce that $
\Phi_\Lambda (FG) = \Phi_\Lambda (F) \star \Phi_\Lambda (G).$

In the same way, we compute
\begin{eqnarray*}
\left[{\hat F}^\Lambda_{h'h}\right]^* (n_1, n_2, n_3) & = & {\overline{{\hat F}^\Lambda_{h'h} (-n_1, -n_2, -n_3)}}\\
& = & \frac{1}{\sqrt{\varphi_{j'}\varphi_j}} \delta_{h'-\hat\varphi_{j'-1}, h-\hat\varphi_j}  {\overline{ {\widetilde F}^\Lambda_{j'j} (-n_1, -n_2, -n_3)}}
\end{eqnarray*}
So forgetting the cocycle $\sigma$ which doesn't perturb the following computation, we can write
\begin{eqnarray*}
{\widetilde F}^\Lambda_{j'j} (-n_1, -n_2, -n_3) & = & \delta_{j_3(j'), j_2^{-n_2}j_1^{-n_1} (j_3(j))} \delta_{-n_3, -k_2^{-n_2}k_1^{-n_1} (j) + \beta_{j'}} F(-n_1, -n_2)\\
 & = & \delta_{j_3(j), j_2^{n_2}j_1^{n_1} (j')} \delta_{n_3, k_2^{-n_2}k_1^{-n_1} (j)- \beta_{j'}} F(-n_1, -n_2)\\
 & = & \delta_{j_3(j), j_2^{n_2}j_1^{n_1} (j')} \delta_{n_3, k_2^{-n_2}k_1^{-n_1} (j_3(j)) +\beta_j- \beta_{j'}} F(-n_1, -n_2)\\
 & = & \delta_{j_3(j), j_2^{n_2}j_1^{n_1} (j')} \delta_{n_3, k_2^{-n_2}k_1^{-n_1} (j_2^{n_2}j_1^{n_1} (j_3(j'))) +\beta_j- \beta_{j'}} F(-n_1, -n_2)\\
 & = & \delta_{j_3(j), j_2^{n_2}j_1^{n_1} (j')} \delta_{n_3, -k_2^{n_2}k_1^{n_1} (j_3(j')) +\beta_j- \beta_{j'}} F(-n_1, -n_2)\\
 & =& \delta_{j_3(j), j_2^{n_2}j_1^{n_1} (j')} \delta_{n_3, -k_2^{n_2}k_1^{n_1} (j') +\beta_j} F(-n_1, -n_2)
\end{eqnarray*}
Hence we see that 
$$
\left[\Phi_{\Lambda} (F)\right]^* = \Phi_{\Lambda} (F^*),
$$
proving finally that $\Phi_\Lambda$ is a $*$-morphism from the involutive algebra $\C[\langle T_1, T_2 \rangle , \sigma]$  to $M_{\hat\varphi_q} (C(\Sigma)\rtimes_{i_*\sigma} \Z^3)$ which is valued in the algebraic crossed product. It is easy to check that $\Phi_\Lambda$ extends to a $C^*$-algebra morphism and hence induces a group morphism
$$
\Phi_{\Lambda, *} : K_*(C^*(\langle T_1, T_2 \rangle , \sigma)) \longrightarrow K_*(C(\Sigma)\rtimes_{i_*\sigma} \Z^3).
$$
Moreover, for any $F\in \C[\langle T_1, T_2 \rangle , \sigma]$, we have
\begin{eqnarray*}
(\tau^\mu\sharp \tr) \Phi_\Lambda (F) & = & \sum_{j-1}^q \sum_{h=\hat\varphi_{j-1}+1}^{\hat\varphi_j}  \tau^\mu \Phi_\Lambda (F)_{hh}\\
& = & \sum_{j=1}^q  \mu (K_{j})  \sum_{h=\hat\varphi_{j-1}+1}^{\hat\varphi_j} {\hat F}^\Lambda_{hh} (0, 0, 0)
\end{eqnarray*}
But $
{\hat F}^\Lambda_{hh} (0, 0, 0) = \frac{1}{\varphi_j} \times F(0, 0)$,
therefore
$$
(\tau^\mu\sharp \tr) \Phi_\Lambda (F)  = \sum_{j=1}^q \mu (K_{j})  F(0, 0) \sum_{h=\hat\varphi_{j-1}+1}^{\hat\varphi_j}   \frac{1}{\varphi_j} = \mu (\Lambda)\times \tau (F).
$$
\end{proof}

We end this section with an explanation of the relation between Theorem \ref{Morphism} and the easy-half of our conjecture in the 3D case.

\begin{proposition}\label{EasyInclusion}
Assume that $p=3$ so that $\Z^3$ acts minimally on $\Sigma$. Assume further that $\Theta_{13}= \Theta_{23}=0$ and that any $\Z$-valued function $\lambda$ on $\Sigma$ which represents a  class in $\left( C(\Sigma, \Z)_{\langle T_3\rangle}\right)^{\Z^3}$ can be decomposed as a finite algebraic sum of characteristic functions of clopen subspaces satisfying the hypothesis (H). Then  the magnetic frequency group $\Z[\mu] + \Theta_{12} \Z_{12}[\mu]$ is contained in the magnetic gap-labelling group associated with the multiplier $\sigma$ which corresponds to $\Theta$. 

\end{proposition}

\begin{proof}\ 
We only need to prove that $ \Theta_{12} \Z_{12}[\mu]$ is contained in the magnetic gap-labelling group. 

We denote by $\sigma$ the multiplier of $\langle T_1, T_2\rangle$ which is associated with the $2\times 2$-matrix 
$$
\left(\begin{array}{cc} 0 & \Theta_{12} \\ -\Theta_{12} & 0  \end{array}\right).
$$ 
Then using the inclusion $i:\langle T_1, T_2\rangle \hookrightarrow \Z^3$ and applying Theorem \ref{Morphism}, we deduce that for any clopen subspace $\Lambda$ which satisfies (H) for $T_3$, 
$$
\{ \mu (\Lambda) \times \tau_*(x), x\in  K_0(C^*(\langle T_1, T_2 \rangle , \sigma), \Lambda \text{ as before}\} \subset \tau^\mu_*(K_0(C(\Sigma)\rtimes_{i_*\sigma_{12}} \Z^3)),
$$
where $i_*\sigma$ is associated with the skew matrix 
$$
\left(\begin{array}{ccc} 0 & \Theta_{12} & 0\\ -\Theta_{12} & 0 & 0\\ 0 & 0 & 0 \end{array}\right).
$$ 
Hence $ \Theta_{12} \Z_{12}[\mu]$ is contained in the magnetic gap-labelling group.\\
\end{proof}

 
\appendix



\section{The coinvariants as a direct summand in $K$-theory}


Here we give a direct proof  that the coinvariants are always direct summands in the $K$-theory of the twisted crossed product $C^*$-algebra, without using the Packer-Raeburn trick \cite{PR}. The method is an important step towards a complete solution to our easy-half conjecture for all dimensions.

Let $\Sigma$ be a Cantor space on which $\ZZ^p$ acts by homeomorphisms. We do not assume that the action is minimal 
in this section, because we want to use induction on $p$. Consider the twisted crossed product $C^*$-algebra
$C(\Sigma)\rtimes_\sigma\ZZ^p$, with $\Theta$ a skew-symmetric $(p \times p)$ matrix determining the multiplier 
$\sigma$. We can write $C(\Sigma)\rtimes_\sigma\ZZ^p=\left(C(\Sigma)\rtimes_{\sigma|}\ZZ^{p-1}\right)\rtimes_{\phi}\ZZ^{(p)}$  
where $\sigma|$ is the restriction of the multiplier $\sigma$ to $\ZZ^{p-1}\times\ZZ^{p-1}$ and $\ZZ^{(p)}$ is the $p$-th copy of $\ZZ$ and
$\phi$ is an automorphism of $C(\Sigma)\rtimes_{\sigma|}\ZZ^{p-1}$ given by
$$
\phi(U_j) = \exp(2\pi \sqrt{-1} \Theta_{jp}) U_j, \qquad j=1, \ldots (p-1).
$$
where $U_1, \ldots U_{(p-1)}$ are unitary automorphisms of $C(\Sigma)$ generating the twisted crossed product,
$C(\Sigma)\rtimes_{\sigma|}\ZZ^{p-1}$. {{More precisely, the map $\phi$ is an automorphism of $C(\Sigma)\rtimes_{\sigma|}\ZZ^{p-1}$ which is given, on elementary elements $g U_{k_1, \cdots, k_{p-1}}$ of $C(\Sigma)\rtimes_{\sigma|}\ZZ^{p-1}$, by the formula
$$
\phi (g U_{k_1, \cdots, k_{p-1}}) := T_p(g)  \exp(2\pi \sqrt{-1} [k_1 \Theta_{p1}+ \cdots + k_{p-1}\Theta_{p\, p-1}]) U_{k_1, \cdots, k_{p-1}}.
$$
Then an easy inspection shows that we recover in this way $C(\Sigma)\rtimes_{\sigma}\ZZ^{p}$ from $C(\Sigma)\rtimes_{\sigma|}\ZZ^{p-1}$ and the $\Z$ action generated by $\phi$. Indeed, the map 
$$
\left(g U_{k_1, \cdots, k_{p-1}}\right) \delta_{k_p} \longmapsto \exp(\pi \sqrt{-1} [k_1 \Theta_{1p}+ \cdots + k_{p-1}\Theta_{p-1\, p}])g U_{k_1, \cdots, k_{p}},
$$
extends to the allowed $C^*$-algebra isomorphism
$$
\left[C(\Sigma)\rtimes_{\sigma|}\ZZ^{p-1} \right]\rtimes \Z \longrightarrow C(\Sigma)\rtimes_{\sigma}\ZZ^{p}.
$$
}} Clearly $\phi$ is homotopic to the identity, since the exponential defining it can be scaled. This is essentially an argument in \cite{Phillips}.

The Pimsner-Voiculescu (PV) sequence \cite{PV} gives the exact sequences for $i=0,1$,
\begin{equation}
0\longrightarrow K_i(C(\Sigma)\rtimes_{\Theta|}\ZZ^{p-1})_{\ZZ^{(p)}}\longrightarrow K_i(C(\Sigma)\rtimes_\Theta\ZZ^p)\overset{\partial_p}{\longrightarrow}K_{i-1}(C(\Sigma)\rtimes_{\Theta|}\ZZ^{p-1})^{\ZZ^{(p)}}\longrightarrow 0,
\end{equation}
where $(\cdot)^{\ZZ^{(p)}}$ and $(\cdot)_{\ZZ^{(p)}}$ denote, respectively, the invariants and coinvariants of $(\cdot)$ under the induced map in $K$-theory of the $\ZZ^{(p)}$ action. We will abbreviate these by $(\cdot)^p$ and $(\cdot)_p$ respectively. 
We can calculate these $K$-theory groups iteratively. Set $\cA=C(\Sigma,\ZZ)$.  Then the following is a consequence of the general spectral sequence which computes $K$-theory if one uses the Packer-Raeburn trick \cite{PR}. We give it an independent treatment here.

\begin{proposition}
In the notation above, 
$\cA_{123\ldots p}$ is a direct summand in $K_0(C(\Sigma)\rtimes_{\sigma}\ZZ^{p})$.
\end{proposition}

\begin{proof}

It is known that 
\beq
    K_0(C(\Sigma))=\cA,\qquad K_1(C(\Sigma))=0.
\eeq

For the first $\ZZ$ action (for which $\Theta|$ vanishes trivially), we have
\begin{eqnarray}
K_0(C(\Sigma)\rtimes\ZZ)&\cong& \cA_1 \\
K_1(C(\Sigma)\rtimes\ZZ)&\cong& \cA^1.
\end{eqnarray}

For the first $\ZZ^2$ action, using the PV sequence we have
\begin{equation}
K_0(C(\Sigma)\rtimes_{\Theta|}\ZZ^2)\cong \cA_{12}\oplus\cA^{12}
\end{equation}
because $\cA^{12}\subset\cA$ is free abelian. 
So, $\cA_{12}$ is a direct summand in $K_0(C(\Sigma)\rtimes_{\Theta|}\ZZ^2)$.

For the first $\ZZ^3$ action, we have
\begin{equation}
0\longrightarrow (\cA_{12}\oplus\cA^{12})_3\longrightarrow K_0(C(\Sigma)\rtimes_{\Theta|}\ZZ^3)\longrightarrow
K_1(C(\Sigma)\rtimes_{\Theta|}\ZZ^2)^3\longrightarrow 0,
\end{equation}
that is
\begin{equation}
0\longrightarrow \cA_{123}\oplus(\cA^{12})_3\longrightarrow K_0(C(\Sigma)\rtimes_{\Theta|}\ZZ^3)\longrightarrow
K_1(C(\Sigma)\rtimes_{\Theta|}\ZZ^2)^3\longrightarrow 0,
\end{equation}
Therefore, $\cA_{123}$ is a direct summand in $K_0(C(\Sigma)\rtimes_{\Theta|}\ZZ^3)$.

Suppose now that $\cA_{123\ldots (p-1)}$ is a direct summand in $K_0(C(\Sigma)\rtimes_{\sigma|}\ZZ^{p-1})$. That is, 
$$
K_0(C(\Sigma)\rtimes_{\sigma|}\ZZ^{p-1}) = \cA_{123\ldots (d-1)} \oplus \cM.
$$

Then the PV sequence implies that 
\begin{equation}
0\longrightarrow \left(\cA_{12\ldots (p-1)}\oplus \cM\right)_p\longrightarrow K_0(C(\Sigma)\rtimes_{\sigma}\ZZ^p)\longrightarrow
K_1(C(\Sigma)\rtimes_{\sigma|}\ZZ^{p-1})^p\longrightarrow 0,
\end{equation}
that is,
\begin{equation}
0\longrightarrow \cA_{12\ldots p}\oplus \cM_p\longrightarrow K_0(C(\Sigma)\rtimes_{\sigma}\ZZ^p)\longrightarrow
K_1(C(\Sigma)\rtimes_{\sigma|}\ZZ^{p-1})^p\longrightarrow 0.
\end{equation}

In particular,
$\cA_{123\ldots p}$ is a direct summand in $K_0(C(\Sigma)\rtimes_{\sigma}\ZZ^{p})$,
thereby proving the Proposition.

\end{proof}


\section{A more detailed history of gap-labelling theorems}\label{history}


We give here a brief overview of the history of the gap-labelling theorems and conjectures for the last 35 years.  We thank Jean Bellissard for his invaluable help concerning this section. 

The first mention of a {\em gap-labelling theorem} probably goes back to 
a paper by J.~Moser in 1981 \cite{Moser}, concerning the Schr\"odinger operator in 1D with an 
almost periodic potential: he proved that the 
gaps are labeled by the {\em frequency module} of the potential, namely the 
$\mathbb{Z}$-module generated by 
the frequencies of the Bohr-Fourier decomposition of the almost periodic potential.
These ideas were further developed by Johnson and Moser \cite{JM}.

In higher dimensions, it turns out that the frequency module doesn't  label spectral gaps,
as seen by a counter-example \cite{Bellissard2}.
Bellissard's version of the {\em gap-labelling theorem} says that the spectral gaps of any self-adjoint 
operator $H$ (bounded or not) are labeled by the $K_0$-group of a
C$^\ast$-algebra this operator is affiliated with. In the case of an 
operator which is {\em homogeneous} (see \cite{Bellissard85}) with respect to some 
translations group $G ={\mathbb R}^p$ or $G={\mathbb Z}^p$, 
this C$^\ast$-algebra can be chosen to be, up to Morita equivalence,  the crossed product algebra 
$C(\Omega)\rtimes G$ where $\Omega$ is a compact space in the strong topology.
In addition, using the Shubin formula (see \cite{Bellissard85}) the 
labels are expressed as the image of $K_0$ by a canonical trace and the 
label of a given gap is also given by the value of the Integrated Density 
of States, see \cite{Bellissard1, Bellissard2, Bellissard85}.

The general proof was finally given by \cite{BO, KP, BBG} under some integrality assumption on the associated \v{C}ech Chern character.
We recall that a 
quasicrystal is a distribution of points in the space obtained by the 
so-called {\em cut-and-project method} (also called {\em model sets} following 
the PhD work of Yves Meyer in 1972). The result concerning a ${\mathbb 
Z}^p$-action on a Cantor set $\Sigma$ is more general and applies also to aperiodic systems 
which are {\em not} cut-and-project, such as the {\em Thue-Morse sequence} or the {\em chair 
tiling}, to cite only a few (for other examples, see  \cite{Sadun}). 
The modern way of computing the gap labels involves the \v{C}ech 
cohomology of the hull, initiated by 
\cite{AP}.

Proposition 6.2.1 in \cite{Bellissard4} has a formula involving determinants that superficially
resembles our Conjecture \ref{mainconj}. We emphasize that it does {\em not} include the magnetic field. The relevant 
matrix in that formula is the {\em frequency matrix} of the aperiodic potential (actually square
submatrices of it). Another assumption made is that $\Omega$ is a manifold, in which case
one is able to get more precise results. In our case, $\Omega$ is never a manifold as it is either a Cantor set $\Sigma$ when $G=\ZZ^p$, or a fibre bundle $X$ over 
the torus with fibre a Cantor set $\Sigma$ when $G=\RR^p$.\\


\end{document}